\pdfoutput=1
\documentclass[12pt]{amsart}                  
\usepackage{latexsym,amsfonts,amssymb,epsfig,graphicx,natbib,setspace,color,verbatim}
\usepackage{amsmath,amsthm}
\usepackage{lipsum}
\usepackage{hyperref}
\usepackage[english]{babel}
\usepackage{natbib}
\usepackage{wrapfig}
\usepackage{tikz}

\usepackage{latexsym}
\begin{document}

\newcommand{\out}{\ensuremath{\mathrm{Out}(\mathbb{F}_N) } }
\newcommand{\aut}{\ensuremath{\mathrm{Aut}(\mathbb{F}_N) }}
\newcommand{\F}{\ensuremath{\mathbb{F }_N } }
\newtheorem{theorem}{Theorem}[section]
\newtheorem{corollary}[theorem]{Corollary}
\newtheorem{lemma}[theorem]{Lemma}
\newtheorem{proposition}[theorem]{Proposition}
\newtheorem{remark}[theorem]{Remark}
\newtheorem{definition}[theorem]{Definition}

\title{Applications of Weak attraction theory in $\out$}



\author{Pritam Ghosh}



\begin{abstract}
Given a finite rank free group $\F$ of rank $\geq 3$ and two exponentially growing outer automorphisms $\psi$ and $\phi$ with dual lamination pairs $\Lambda^\pm_\psi$ and $\Lambda^\pm_\phi$ associated to them, 
which satisfy a notion of independence described in this paper, 
    we will use the pingpong techniques developed by Handel and Mosher \citep{HM-13c} to show that there exists an integer $M > 0$, 
    such that for every $m,n\geq M$, the group $G=\langle \psi^m,\phi^n \rangle $ will be a free group of rank two and 
    every element of this free group which is not conjugate to a power of the generators will be fully irreducible and hyperbolic. 

\end{abstract}

\maketitle

\section{Introduction}
\label{intro}

Let $\F$ be a free group of rank $N\geq 3$. The quotient group 

$\aut/\mathrm{Inn}(\F)$, denoted by $\out$,  is called the group  of outer automorphisms of $\F$. There are many tools in studying the properties of this group. One of 
them is by using train-track maps introduced in \citep{BH-92} and later generalized in \citep{BFH-97}, \citep{BFH-00}, \citep{FH-11}. The fully-irreducible 
outer automorphisms are the most well understood elements in $\out$ . They behave very closely to the pseudo-Anosov homeomorphisms of surfaces with one boundary component, 
which have been well understood and are a rich source of examples and interesting theorems. We however, 
will focus on exponentially growing outer automorphisms which might not be fully irreducible but exhibit some properties similar to fully-irreducible elements.
It was shown by Bestvina-Feighn\citep{BF-92} and Brinkmann \citep{Br-00} that an outer automorphism was \textit{hyperbolic} if and only if it did not have any periodic conjugacy classes. 
This is an interesting class of elements of $\out$ as we will 
see later.

We introduce a notion of pairwise independence (see definition \ref{independent}) for exponentially growing elements $\psi,\phi \in \out$  equipped with some dual lamination pairs $\Lambda^\pm_\psi,\Lambda^\pm_\phi$, respectively. 
 The definition of pairwise independence that we see here is an abstraction of properties that independent pseudo-Anosovs enjoy in case of mapping class groups. Namely, $\Lambda^\epsilon_\phi$ is left invariant by $\phi^\epsilon$, 
 where $\epsilon = +,-$ (similarly for $\psi$) and that $\{\Lambda^\pm_\psi\} \cup \{\Lambda^\pm_\phi\}$ fills and are attracted to 
each other under iteration of the appropriate automorphism and there does not 
exist any nontrivial conjugacy class that is neither attracted to $\Lambda^+_\phi$ nor to $\Lambda^+_\psi$. In short, we filter out some important dynamical behavior of these automorphisms that we believe should be obvious plus add some conditions 
on geometricity properties of laminations of the elements of subgroup generated by $\psi$ and $\phi$.  
More discussions on this follows after the statement of the theorem.

We now state our main theorem:

\vspace{0.2 cm}
\textbf{ Theorem \ref{main}} \textit{Let $\phi , \psi $ be two  exponentially growing elements of $\out$, such that there exists some dual lamination pairs $\Lambda^\pm_\phi$ and $\Lambda^\pm_\psi$, which makes $\phi$, $\psi$ pairwise independent. 
Then there exists an $M\geq0$, 
such that for all $n,m\geq M$ the group $\langle\psi^m,\phi^n\rangle$ will be free of rank two and every element of this free group, 
not conjugate (in $G$) to some power of the generators, will be hyperbolic and fully-irreducible.}
\vspace{0.4 cm}

A corollary of the main theorem is a theorem of Kapovich-Lustig.

 (Theorem 6.2, \citep{KL-10}) If $\phi,\psi\in\out$ are fully irreducible hyperbolic elements so that if $\langle \phi, \psi \rangle$ is not virtually cyclic, then there exits $m,n\geq 1$ so that $\langle \psi^m,\phi^n \rangle$ 
 is a free group of rank two and every nontrivial element of the group is also fully irreducible and hyperbolic.

 With the developments that have been made recently, the above result can be obtained by a geometric pingpong argument. It follows from the work of Bestvina-Feighn \citep{BF-12}  and Mann \citep{Mann}.
 The loxodromic elements of the Bestvina-Feighn Free factor complex are the fully irreducible outer automorphisms, and the loxodromic elements of Mann's ``co-surface'' graph are the hyperbolic fully-irreducible 
 outer automorphisms. However, neither the main theorem of this paper nor any of it's corollaries, except for the one mentioned above, can be fully obtained by any geometric pingpong arguments as of now.

 Another interesting theorem in somewhat similar spirits was obtained by Clay and Pettet in \citep{CP-10} and by Gultepe \citep{Gul}. They show that for sufficiently independent Dehn-twists in $\out$, they generate a free group 
all whose elements, not conjugate to some power of the generators, are hyperbolic and fully irreducible. Their proofs are however very different from each other. Gultepe uses geometric pingpong on free factor complex and 
co-surface graph whereas Clay and Pettet use Guirardel's core and more tree-centric arguments.
 
The fact that sufficiently high powers of independent fully irreducible outer automorphisms generate a free group of rank two was shown first in Proposition 3.7 of \citep{BFH-97}. 
In the theorem of Kapovich and Lustig the generating outer automorphisms are assumed to be fully irreducible and hyperbolic, but our assumptions are much weaker. Not only does their theorem follow from ours, but by our \textbf{Corollary} \ref{app} and 
\textbf{Theorem} \ref{anosov} we give a complete case analysis of how all the elements of such a free group of rank two will look like,
depending on the nature of $\phi,\psi$ - not just the case when both are hyperbolic. One key difference between the work of Kapovich-Lustig and this paper is the tools that have been used. They prove their results using 
\textit{geodesic currents in free groups} where as we use train-track theory.

One of the key motivations behind this problem, other than studying the dynamics of exponentially growing elements, is to use it for understanding the correct notion 
of convex cocompactness in $\out$. This was inspired by the following theorem of Farb-Mosher


\textbf{Theorem:}(Theorem 1.4, \citep{FaM-02}) If $\phi_1,......,\phi_k \in \mathcal{MCG}(S)$ are pairwise independent pseudo-Anosov elements, then for all sufficiently large positive integers $a_1,.....,a_k$ 
the mapping classes $\phi^{a_1}_1,......,\phi^{a_k}_k $ 
will freely generate a Schottky subgroup $F$ of $\mathcal{MCG}(S)$ .

The notion of independence of the pseudo-Anosov elements in 

$\mathcal{MCG}(\mathcal{S})$ is equivalent to the property that the attracting and repelling laminations of the two elements are 
mutually transverse on the surface, 
from which it follows that the collection of laminations fills (in fact they individually fill). 
These filling properties are enjoyed by fully irreducible elements in $\out$. 
But exponentially growing elements which are not fully irreducible might not have this property. In the definition if \textquotedblleft pairwise independence \textquotedblright \ref{independent} 
we list the properties similar to pseudo-Anosov maps that make the aforementioned theorem work.
As a consequence of the techniques developed to prove Theorem \ref{main}, we also prove the following theorem which is very similar to the theorem of Farb and Mosher, 
except that in our case we have surfaces with one boundary component.

\vspace{0.2cm}

\textbf{Theorem \ref{anosov}} Let $\mathcal{S}$ be a connected, compact surface (not necessarily oriented) with one boundary component. Let $f,g \in \mathcal{MCG(S)}$ be psedo-Anosov 
homeomorphisms of the surface which are not conjugate to powers of each other. 
Then there exists some integer $M$ such that the group $G= \langle f^m, g^n \rangle$ will be free for every $m,n> M$, 
and every element of this group will be a pseudo-Anosov. 

The techniques we use to prove our theorems were developed in \citep{BH-92}, \citep{BFH-97}, \citep{BFH-00}, \citep{FH-11}, \citep{HM-09} and \citep{HM-13c}. Given pairwise independent exponentially growing elements of $\out$, 
we use pingpong type argument developed in \citep{HM-13c} to produce exponentially growing elements which are hyperbolic. 
Then we proceed to show that these elements will be fully irreducible by using Stallings graphs (defined in section \ref{sec:11}).
Finally we prove Theorem \ref{anosov} (which is a theorem about surface homeomorphisms) as an application of the tools developed here to study $\out$.

\textbf{Acknowledgment:} This work was a part of author's Ph.D thesis at Rutgers University, Newark. 
The author would like to thank his advisor Dr. Lee Mosher for his continuous encouragement and support. This work was partially supported by Dr Lee Mosher's NSF grant.

\section{Preliminaries}
\label{sec:1}
In this section we give the reader a short review of the defintions and some important results in $\out$ which are revelant to the theorem that we are trying to prove here. All the results which are stated as facts, can be found in full details in \citep{BFH-00}, \citep{FH-11}, \citep{HM-13c}. 

\subsection{Weak topology}
\label{sec:2}
Given any finite graph $R$, let $\widehat{\mathcal{B}}(G)$ denote the compact space of equivalence classes of circuits in $G$ and paths in $G$, whose endpoints (if any) are vertices of $G$. We give this space the \textit{weak topology}.
Namely, for each finite path $\gamma$ in $G$, we have one basis element $\widehat{N}(G,\gamma)$ which contains all paths and circuits in $\widehat{\mathcal{B}}(G)$ which have $\gamma$ as its subpath. 
Let $\mathcal{B}(G)\subset \widehat{\mathcal{B}}(G)$ be the compact subspace of all lines in $G$ with the induced topology. One can give an equivalent description of $\mathcal{B}(G)$ following \citep{BFH-00}.
A line is completely determined, upto reversal of direction, by two distinct points in $\partial \F$, since there only one line that joins these two points. We can then induce the weak topology on the set of lines coming from the Cantor set $\partial \F$. More explicitly,
let $\widetilde{\mathcal{B}}=\{ \partial \F \times \partial \F - \vartriangle \}/(\mathbb{Z}_2)$, where $\vartriangle$ is the diagonal and $\mathbb{Z}_2$ acts by interchanging factors. We can put the weak topology on 
$\widetilde{\mathcal{B}}$, induced by Cantor topology on $\partial \F$. The group $\F$ acts on $\widetilde{\mathcal{B}}$ with a compact but non-Hausdorff quotient space $\mathcal{B}=\widetilde{\mathcal{B}}/\F$. 
The quotient topology is also called the \textit{weak topology}. Elements of $\mathcal{B}$ are called \textit{lines}. A lift of a line $\gamma \in \mathcal{B}$ is an element  $\widetilde{\gamma}\in \widetilde{\mathcal{B}}$ that 
projects to $\gamma$ under the quotient map and the two elements of $\widetilde{\gamma}$ are called its endpoints. 

One can naturally identify the two spaces $\mathcal{B}(G)$ and $\mathcal{B}$ by considering a homeomorphism between the two Cantor sets $\partial \F$ and set of ends of universal cover of $G$ , where $G$ is a marked graph.
$\out$ $\curvearrowright \mathcal{B}$. The actions comes from the action of Aut($\F$) on $\partial \F$. Given any two marked graphs $G,G'$ and a homotopy equivalence $f:G\rightarrow G'$ between them, the induced map 
$f_\#: \widehat{\mathcal{B}}(G)\rightarrow \widehat{\mathcal{B}}(G')$ is continuous and the restriction $f_\#:\mathcal{B}(G)\rightarrow \mathcal{B}(G')$ is a homeomorphism. With respect to the identification 
$\mathcal{B}(G)\approx \mathcal{B}\approx \mathcal{B}(G')$, if $f$ preserves the marking then $f_{\#}:\mathcal{B}(G)\rightarrow \mathcal{B}(G')$ is equal to the identity map on $\mathcal{B}$. When $G=G'$, $f_{\#}$ agree with their
 homeomorphism $\mathcal{B}\rightarrow \mathcal{B}$ induced by the outer automorphism associated to $f$.
 
A line(path) $\gamma$ is said to be \textit{weakly attracted} to a line(path) $\beta$ under the action of $\phi\in\out$, if the $\phi^k(\gamma)$ converges to $\beta$ in the weak topology. This is same as saying, for any given finite subpath of $\beta$, $\phi^k(\gamma)$ 
contains that subpath for some value of $k$; similarly if we have a homotopy equivalence $f:G\rightarrow G$,  a line(path) $\gamma$ is said to be \textit{weakly attracted} to a line(path) $\beta$ under the action of $f_{\#}$ if the $f_{\#}^k(\gamma)$ converges to $\beta$
 in the weak topology. The \textit{accumulation set} of a ray $\gamma$ in $G$ is the set of lines  $l\in \mathcal{B}(G)$ which are elements of the weak closure of $\gamma$; which is same as saying every finite subpath of $l$ 
occurs infinitely many times as a subpath $\gamma$.

\subsection{Free factor systems and subgroup systems}
\label{sec:3}
Given a finite collection $\{K_1, K_2,.....,K_s\}$ of subgroups of $\F$ , we say that this collection determines a \textit{free factorization} of $\F$ if $\F$ is the free product of these subgroups, that is, 
$\F = K_1 * K_2 * .....* K_s$. The conjugacy class of a subgroup is denoted by [$K_i$].

A \textit{free factor system} is a finite collection of conjugacy classes of subgroups of $\F$ , $\mathcal{K}:=\{[K_1], [K_2],.... [K_p]\}$ such that there is a free factorization of $\F$ of the form 
$\F = K_1 * K_2 * ....* B$, where $B$ is some finite rank subgroup of $\F$ (it may be trivial). Given two free factor systems $\mathcal{K}, \mathcal{K}'$ we give a partial ordering to set of all free factor systems by defining 
$\mathcal{K}\sqsubset \mathcal{K}'$ if for each conjugacy class of subgroup $[K]\in \mathcal{K}$ there exists some conjugacy class of subgroup $[K']\in \mathcal{K}'$ such that $K$ is a free factor of $K'$.

There is an action of $\out$ on the set of all conjugacy classes of subgroups of $\F$. This action induces an action of $\out$ on the set of all free factor systems. For notation simplicity we will avoid writing $[K]$ all the time and write $K$ instead, when we discuss   
the action of $\out$ on this conjugacy class of subgroup $K$ or anything regarding the conjugacy class [$K$]. It will be understood that we actually mean [$K$].
\begin{lemma}
 [\citep{BFH-00}, Section 2.6] Every collection $\{\mathcal{K}_i\}$ of free factor systems has a well-defined meet $\wedge\{\mathcal{K}_i\} $, which is the unique maximal free factor system $\mathcal{K}$ such that $\mathcal{K}\sqsubset \mathcal{K}_i$ for all $i$. Moreover, 
 for any free factor $F< \F$ we have $[F]\in \wedge\{\mathcal{K}_i\}$ if and only if there exists an indexed collection of subgroups $\{F_i\}_{i\in I}$ such that $[A_i]\in \mathcal{K}_i$ for each $i$ and $A=\bigcap_{i\in I} A_i$. 
\end{lemma}

For any marked graph $G$ and any subgraph $H \subset G$, the fundamental groups of the noncontractible components of $H$ form a free factor system . We denote this by $[H]$. A subgraph of $G$ which has no valence 1 vertex is called a \textit{core graph}. 
Every subgraph has a unique core graph, which is a deformation retract of its noncontractible components.
A free factor system $\mathcal{K}$ \textit{carries a conjugacy class} $[c]$ in $\F$ if there exists some $[K] \in \mathcal{K}$ such that $c\in K$. We say that $\mathcal{K}$ \textit{carries the line} $\gamma \in \mathcal{B}$ if for any marked graph $G$ 
the realization of $\gamma$ in $G$ is the weak limit of a sequence of circuits in $G$ each of which is carried by $\mathcal{K}$.
An equivalent way of saying this is: for any marked graph $G$ and a subgraph $H \subset G $ with $[H]=\mathcal{K}$, the realization of $\gamma$ in $G$ is contained in $H$.

Similarly define a \textit{subgrpoup system} $\mathcal{A} = \{[H_1], [H_2], .... ,[H_k]\}$ to be a finite collection of conjugacy classes of finite rank subgroups $H_i<\F$. Define a subgroup system to be \textit{malnormal} if for any $[H_i],[H_j]\in \mathcal{A}$, 
if $H_i^x\cap H_j$ is nontrivial then $i=j$ and $x\in H_i$. Two subgroup systems $\mathcal{A}$ and $\mathcal{A}'$ are said to be \textit{\textbf{mutually malnormal}} if both $H_i^x\cap H'_j$ and $H_i\cap (H'_j)^x$ are trivial for every 
$[H_i]\in \mathcal{A}, [H'_j]\in \mathcal{A}'$ and $x\in \F$

A subgroup system $\mathcal{A}$ carries a conjugacy class $[c]\in \F$ if there exists some $[A]\in\mathcal{A}$ such that $c\in A$. Also, we say that $\mathcal{A}$ carries a line $\gamma$ if one of the following equivalent conditions hold:
\begin{itemize}
 \item $\gamma$ is the weak limit of a sequence of conjugacy classes carries by $\mathcal{A}$.
 \item There exists some $[A]\in \mathcal{A}$ and a lift $\widetilde{\gamma}$ of $\gamma$ so that the endpoints of $\widetilde{\gamma}$ are in $\partial A$.
 
\end{itemize}
The following fact is an important property of lines carried by a subgroup system. The proof is by using the observation that $A<\F$ is of finite rank implies that $\partial A$ is a compact subset of $\partial \F$
\begin{lemma}
 For each subgroup system $\mathcal{A}$ the set of lines carried by $\mathcal{A}$ is a closed subset of $\mathcal{B}$
\end{lemma}

So by definition, a line $\gamma$ is carried by a conjugacy class (circuit) $\rho$ means that $\gamma$ is a bi-infinite iteration of $\rho$, denoted by $\rho_\infty$.

From \citep{BFH-00}
The \textit{free factor support} of a set of lines $B$ in $\mathcal{B}$ is (denoted by $\mathcal{A}_{supp}(B)$) defined as the meet of all free factor systems that carries $B$. If $B$ is a single line then $\mathcal{A}_{supp}(B)$ is 
single free factor. We say that a set of lines, $B$, is \textit{filling} if $\mathcal{A}_{supp}(B)=[\F]$

\subsection{Principal automorphisms and rotationless automorphisms}
\label{sec:4}

Given an outer automorphism $\phi\in\out$, we can consider a lift $\Phi$ in $\aut$. We call a lift principal automorphism, if it satisfies certain conditions described below. 
Roughly speaking, what such lifts guarantees is the the existence of certain lines which are not a part of the attracting lamination but it still 
fills the free group $\F$. Such lines (called \textit{singular lines}) will be a key tool in describing the set of lines which are not attracted to the attracting lamination of $\phi$.\\
Consider $\phi\in \out$ and a lift $\Phi$ in $\aut$. $\Phi$ has an action on $\F$, which has a fixed subgroup denoted by Fix$(\Phi)$. Consider the boundary of this fixed subgroup $\partial$Fix$(\Phi)\subset$ $\partial \F$. It is either empty or has exactly two points.\\
This action action extends to the boundary and is denoted by $\widehat{\Phi}:\partial \F\rightarrow\partial \F$. Let Fix$(\widehat{\Phi})$ denote the set of fixed points of this action. We call an element $P$ of Fix$(\widehat{\Phi})$ \textit{attracting fixed point} if there exists an open neighborhood $U\subset \partial \F$ of $P$ such that we have $\widehat{\Phi}(U)\subset U$ 
and for every points $Q\in U$ the sequence $\widehat{\Phi}^n(Q)$ converges to $P$. Let Fix$_+(\widehat{\Phi})$ denote the set of attracting fixed points of Fix$(\widehat{\Phi})$. Similarly let Fix$_-(\widehat{\Phi})$ denote the attracting fixed points of Fix$(\widehat{\Phi}^{-1})$. \\
Let Fix$_N(\widehat{\Phi})=$ Fix$(\widehat{\Phi})$ $\cup$ Fix$_-(\widehat{\Phi})= \partial$Fix$(\Phi)$ $\cup$ Fix$_+(\widehat{\Phi})$. We say that an automorphism $\Phi\in\aut$ in the outer automorphism class of $\phi$ is a \textit{principal automorphism} if Fix$_N(\widehat{\Phi})$ has at least 3 points or Fix$_N(\widehat{\Phi})$ has exactly two points which are neither the endpoints of an axis of a covering translation, nor the endpoints of a 
 generic leaf of an attracting lamination $\Lambda^+_\phi$. The set of all principal automorphisms of $\phi$ is denoted by $P(\phi)$. This set is invariant under isogredience.\\
We then have the following lemma from [\citep{GJLL-98}] and \citep{HM-09}:\\
\begin{lemma}
\label{fixedpoints}
 If $\phi\in \out$ is fully irreducible and $\Phi$ is a principal automorphism representing $\phi$, then:
\begin{enumerate}
 \item If Fix$(\Phi)$ is trivial then Fix$_N(\widehat{\Phi})$ is a finite set of attractors.
\item If Fix$(\Phi)=\langle\gamma\rangle$ is infinite cyclic, then Fix$_N(\widehat{\Phi})$ is the union of the endpoints of the axis of the covering translation $t_\gamma^\pm$ with a finite set of $t_\gamma$-orbits of attractors.
\item If $P\in$Fix$_N(\widehat{\Phi})$ is an attractor then it is not the end points of an axis of any covering translation $t_\gamma$.
\end{enumerate}

\end{lemma}

Let Per$(\widehat{\Phi})$ = $\cup_{k\geq1}$Fix$(\widehat{\Phi}^k)$, Per$_+(\widehat{\Phi})$ = $\cup_{k\geq1}$Fix$_+(\widehat{\Phi}^k)$ and similarly define Per$_-(\widehat{\Phi})$ and Per$_N(\widehat{\Phi})$. \\
We say that $\phi\in \out$ is rotationless if Fix$_N(\widehat{\Phi})$ = Per$_N(\widehat{\Phi})$ for all $\Phi\in P(\phi)$, and if for each $k\geq1$ the map $\Phi\rightarrow \Phi^k$ induces a bijection between $P(\phi)$ and $P(\phi^k)$. \\
The following two important facts about rotationless automorphisms are taken from \citep{FH-11}. The following fact in particular is heavily used. Whenever we write \textquotedblleft pass to a rotationless power \textquotedblright we intend to use this uniform constant $K$ given by the fact.
\begin{lemma}[\citep{FH-11},Lemma 4.43]
\label{rotationless}
 There exists a $K$ depending only upon the rank of the free group $\F$ such that for every $\phi\in \out$ , $\phi^K$ is rotationless.
\end{lemma} 

\begin{lemma}[\citep{FH-11}]
 If $\phi\in \out$ is rotationless then:
\begin{itemize}
 \item Every periodic conjugacy class of $\phi$ is a fixed conjugacy class.
\item Every free factor system which is periodic under $\phi$ is fixed.
\end{itemize}

\end{lemma}

\subsection{Topological representatives and Train track maps}
\label{sec:5}
Given $\phi\in\out$ a \textit{topological representative} is a homotopy equivalence $f:G\rightarrow G$ such that $\rho: R_r \rightarrow G$ is a marked graph, $f$ takes vertices to vertices and edges to paths and $\overline{\rho}\circ f \circ \rho: R_r \rightarrow R_r$ represents $R_r$. A nontrivial path $\gamma$ in $G$ is a \textit{periodic Nielsen path} if there exists a $k$ such that $f^k_\#(\gamma)=\gamma$; 
the minimal such $k$ is called the period and if $k=1$, we call such a path \textit{Nielsen path}. A periodic Nielsen path is \textit{indivisible} if it cannot be written as a concatenation of two or more nontrivial periodic Nielsen paths. 

Given a subgraph $H\subset G$ let $G\setminus H$ denote the union of edges in $G$ that are not in $H$.

Given a marked graph $G$ and a homotopy equivalence $f:G\rightarrow G$ that takes edges to paths, one can define a new map $Tf$ by setting $Tf(E)$ to be the first edge in the edge path associated to $f(E)$; similarly let $Tf(E_i,E_j) = (Tf(E_i),Tf(E_j))$. So $Tf$ is a map that takes turns to turns. We say that a nondegenerate turn is illegal if for some iterate of $Tf$ the turn becomes degenerate; otherwise the 
 turn is legal. A path is said to be legal if it contains only legal turns and it is $r-legal$ if it is of height $r$ and all its illegal turns are in $G_{r-1}$.
 
 \textbf{Relative train track map.} Given $\phi\in \out$ and a topological representative $f:G\rightarrow G$ with a filtration $G_0\subset G_1\subset \cdot\cdot\cdot\subset G_k$ which is preserved by $f$, we say that $f$ is a train relative train track map if the following conditions are satisfied:
 \begin{enumerate}
  \item $f$ maps r-legal paths to legal r-paths.
  \item If $\gamma$ is a path in $G$ with its endpoints in $H_r$ then $f_\#(\gamma)$ has its end points in $H_r$.
  \item If $E$ is an edge in $H_r$ then $Tf(E)$ is an edge in $H_r$
 \end{enumerate}

 For any topological representative $f:G\rightarrow G$ and exponentially growing stratum $H_r$, let $N(f,r)$ be the number of indivisible Neilsem paths $\rho\subset G$ that intersect the interior of $H_r$. Let $N(f)= \Sigma_r N(f,r)$. Let $N_{min}$ be the minimum value of $N(f)$ that occurs among the topological representatives with $\Gamma=\Gamma_{min}$. We call a relative train track map stable if $\Gamma=\Gamma_{min}$ and $N(f)=N_{min}$.
 The following result is Theorem 5.12 in \citep{BH-92} which assures the existence of a stable relative train track map.
 \begin{lemma}
  Every $\phi\in \out$ has a stable relative train track representative.
 \end{lemma}
 If $\phi\in \out$ is fully irreducible then the above fact implies that there exists a stable train track representative for $\phi$.

 The reason why stable relative train track maps are more useful is due to the following fact 
 \begin{lemma} (Theorem 5.15, \citep{BH-92})
 If $f:G\rightarrow G$ is a stable relative train track representative of $\phi\in \out$, and $H_r$ is an exponentially growing stratum, then there exists at most one indivisible Nielsen path $\rho$ of height $r$. If such a $\rho$ exists, then the illegal turn of $\rho$ is the only illegal turn in $H_r$ and $\rho$ crosses every edge of $H_r$.
  \end{lemma}

 \textbf{Splittings, complete splittings and CT's}. Given relative train track map $f:G\rightarrow G$, splitting of a line, path or a circuit $\gamma$ is a decomposition of $\gamma$ into subpaths $....\gamma_0\gamma_1 .....\gamma_k....  $ such that for all $i\geq 1$ the path $f^i_\#(\gamma) =  .. f^i_\#(\gamma_0)f^i_\#(\gamma_1)...f^i_\#(\gamma_k)...$
 The terms $\gamma_i$ are called the \textit{terms} of the splitting of $\gamma$. 
 
 Given two linear edges $E_1,E_2$ and a root-free closed Nielsen path $\rho$ such that $f_\#(E_i) = E_i.\rho^{p_i}$ then we say that $E_1,E_2$ are said to be in the \textit{same linear family} and any path of the form $E_1\rho^m\overline{E}_2$ for some integer $m$ is called an \textit{exceptional path}. 
 
 \textbf{Complete splittings:} A splitting of a path or circuit $\gamma = \gamma_1\cdot\gamma_2......\cdot \gamma_k$ is called complete splitting if each term $\gamma_i$ falls into one of the following categories:
 \begin{itemize}
  \item $\gamma_i$ is an edge in some irreducible stratum.
  \item $\gamma_i$ is an indivisible Nielsen path.
  \item $\gamma_i$ is an exceptional path.
  \item $\gamma_i$ is a maximal subpath of $\gamma$ in a zero stratum $H_r$ and $\gamma_i$ is taken.
 \end{itemize}

 \textbf{Completely split improved relative train track maps}. A \textit{CT} or a completely split improved relative train track maps are topological representatives with particularly nice properties. But CTs do not exist for all outer automorphisms. Only the rotationless outer automorphisms are guranteed to have a CT representative 
 as has been shown in the following Theorem from \citep{FH-11}(Theorem 4.28).
 \begin{lemma}
  For each rotationless $\phi\in \out$ and each increasing sequence $\mathcal{F}$ of $\phi$-invariant free factor systems, there exists a CT $f:G\rightarrow G$ that is a topological 
  representative for $\phi$ and $f$ realizes $\mathcal{F}$.
 \end{lemma}

 The following properties are used to define a CT in \citep{FH-11}. There are actually nine properties. But we will state only the ones we need. The rest are not directly used here but they are all part of the proof of various propositions and lemmas we will be needing and which we have stated here as facts.
 \begin{enumerate}
  \item \textbf{(Rotationless)} Each principal vertex is fixed by $f$ and each periodic direction at a principal vertex is fixed by $Tf$.
  \item \textbf{(Completely Split)} For each edge $E$ in each irreducible stratum, the path $f(E)$ is completely split. 
  \item \textbf{(vertices)} The endpoints of all indivisible Nielsen paths are vertices. The terminal endpoint of each nonfixed NEG edge is principal.
  \item \textbf{(Periodic edges)} Each periodic edge is fixed.
 \end{enumerate}

CTs have very nice properties. The reader can look them up \citep{FH-11} for a detailed exposition or \citep{HM-09} for a quick reference. We list below only a few of them that is needed for us.
\begin{lemma}(\citep{FH-11}, Lemma 4.11)
 A completely split path or circuit has a unique complete splitting.
\end{lemma}

\begin{lemma}
 If $\sigma$ is a finite path or a circuit with endpoint in vertices, then $f^k_\#(\sigma)$ is completely split for all sufficiently large $k\geq 1$.
\end{lemma}

\begin{lemma}
 Every periodic Nielsen path is fixed. Also, for each EG stratum $H_r$ there exists atmost one indivisible Nielsen path of height $r$, upto reversal of oreintation. 
\end{lemma}

\subsection{Attracting Laminations and their properties under CTs}
\label{sec:6}
For any marked graph $G$, the natural identification $\mathcal{B}\approx \mathcal{B}(G)$ induces a bijection between the closed subsets of $\mathcal{B}$ and the closed subsets of $\mathcal{B}(G)$. A closed 
subset in any of these two cases is called a \textit{lamination}, denoted by $\Lambda$. Given a lamination $\Lambda\subset \mathcal{B}$ we look at the corresponding lamination in $\mathcal{B}(G)$ as the 
realization of $\Lambda$ in $G$. An element $\lambda\in \Lambda$ is called a \textit{leaf} of the lamination.\\
A lamination $\Lambda$ is called an \textit{attracting lamination} for $\phi$ is it is the weak closure of a line $l$ (called the \textit{generic leaf of $\lambda$}) satisfying the following conditions:
\begin{itemize}
 \item $l$ is bi-recurrent leaf of $\Lambda$.
\item $l$ has an \textit{attracting neighborhood} $V$, in the weak topology, with the property that every line in $V$ is weakly attracted to $l$.
\item no lift $\widetilde{l}\in \mathcal{B}$ of $l$ is the axis of a generator of a rank 1 free factor of $\F$ .
\end{itemize}

We know from \citep{BFH-00} that with each $\phi\in \out$ we have a finite set of laminations $\mathcal{L}(\phi)$, called the set of \textit{attracting laminations} of $\phi$, and the set $\mathcal{L}(\phi)$ is 
invariant under the action of $\phi$. When it is nonempty $\phi$ can permute the elements of $\mathcal{L}(\phi)$ if $\phi$ is not rotationless. For rotationless $\phi$ $\mathcal{L}(\phi)$ is a fixed set. Attracting laminations are directly related to EG stratas. The following fact is a result from \citep{BFH-00} section 3.

\textbf{Dual lamination pairs. }
We have already seen that the set of lines carried by a free factor system is a closed set and so, together with the fact that the weak closure of a generic leaf $\lambda$ of an attracting lamination $\Lambda$ is the whole lamination $\Lambda$ tells us that 
$\mathcal{A}_{supp}(\lambda) = \mathcal{A}_{supp}(\Lambda)$. In particular the free factor support of an attracting lamination $\Lambda$ is a single free factor.
Let $\phi\in \out$ be an outer automorphism and $\Lambda^+_\phi$ be an attracting lamination of $\phi$ and $\Lambda^-_\phi$ be an attracting lamination of $\phi^{-1}$. We say that this lamination pair is a \textit{dual lamination pair} if $\mathcal{A}_{supp}(\Lambda^+_\phi) = \mathcal{A}_{supp}(\Lambda^-_\phi)$. 
By Lemma 3.2.4 of \citep{BFH-00} there is bijection between $\mathcal{L}(\phi)$ and $\mathcal{L}(\phi^{-1})$ induced by this duality relation. 
The following fact is Lemma 2.35 in \citep{HM-13c}; it establishes an important property of lamination pairs in terms of inclusion. We will use it in proving duality for the attracting and repelling laminations we produce in Proposition \ref{pingpong}. 
\begin{lemma}
\label{lam_incl}
 If $\Lambda^\pm_i, \Lambda^\pm_j$ are two dual lamination pairs for $\phi\in \out$  then $\Lambda^+_i\subset \Lambda^+_j$ if and only if $\Lambda^-_i\subset \Lambda^-_j$.
\end{lemma}

\subsection{Nonattracting subgroup system } 
\label{sec:7}
The \textit{nonattracting subgroup system} of an attractting lamination contains information about lines and circuits which are not attracted to the lamination. The definition of this subgroup system 
is slighlty complicated. So we will leave it to the reader to look it up in \citep{HM-13c} where it is explored in details. We however list some key properties which we will be using.
\begin{lemma}(\citep{HM-13c}- Lemma 1.5, 1.6)
\label{NAS}
 \begin{enumerate}
  \item The set of lines carried by $\mathcal{A}_{na}(\Lambda^+_\phi)$ is closed in the weak topology.
  \item A conjugacy class $[c]$ is not attracted to $\Lambda^+_\phi$ if and only if it is carried by $\mathcal{A}_{na}(\Lambda^+_\phi)$.
  \item $\mathcal{A}_{na}(\Lambda^+_\phi)$ does not depend on the choice of the CT representing $\phi$. 
  \item  Given $\phi, \phi^{-1} \in \out$ both rotationless elements and a dual lamination pair $\Lambda^\pm_\phi$ we have $\mathcal{A}_{na}(\Lambda^+_\phi)= \mathcal{A}_{na}(\Lambda^-_\phi)$
  \item $\mathcal{A}_{na}(\Lambda^+_\phi)$ is a free factor system if and only if the stratum $H_r$ is not geometric.
  \item $\mathcal{A}_{na}(\Lambda^+_\phi)$ is malnormal.
  \item If $\{\gamma_n\}_{n\in\mathbb{N}}$ is a sequence of lines such that every weak limit of every subsequence of $\{\gamma_n\}$ is carried by $\mathcal{A}_{na}(\Lambda_\phi)$ then $\{\gamma_n\}$ is carried by $\mathcal{A}_{na}(\Lambda_\phi)$ for all sufficiently large $n$

 \end{enumerate}

\end{lemma}

\begin{corollary}
 Let $\phi\in \out$ be a fully irreducible geometric element which is induced by a pseudo-Anosov homeomorphism of the surface $\mathcal{S}$ with one boundary component. Let $[c]$ be a conjugacy class representing $\partial \mathcal{S}$.
 Then \\ $\mathcal{A}_{na}(\Lambda^+_\phi) = [\langle c \rangle]$.
\end{corollary}

\begin{proof}
 The surface homeomorphism leaves $\partial \mathcal{S}$ implies that $[c]$ is $\phi$ periodic. By passing to a rotationless power we may assume that $[c]$ is fixed by $\phi$. $\phi$ being fully irreducible implies $\mathcal{A}_{na}(\Lambda^+_\phi) = [\langle c \rangle]$. 
\end{proof}

\section{Singular lines, Extended boundary and Weak attraction theorem}
\label{sec:8}
In this section we will look at some results from \citep{HM-13c} which analyze and identify the set of lines which are not weakly attracted to an attracting lamination $\Lambda^\pm_\phi$, given some exponentially growing element in $\out$. Most of the results stated here are in terms of rotationless elements as in the original work. However, we note that being weakly attracted 
to a lamination $\Lambda_\phi$ is not dependent on whether the element is rotationless. All facts stated here about rotationless elements also hold for non rotationless elemenst also, unless otherwise mentioned. This has been pointed out in Remark 5.1 in \citep{HM-13c}
The main reason for using rotationless elements is to make use of the train track structure from the CT theory. We will use some of the facts to prove lemmas about non rotationless elements which we will need later on.

Denote the set of lines not attracted to $\Lambda^+_\phi$ by $\mathcal{B}_{na}(\Lambda^+_\phi)$. The non-attracting subgroup system carries partial information about such lines as we can see in Lemma \ref{NAS}. Other obvious lines which are not attracted are the generic leaves of $\Lambda^-_\phi$. There is another class of lines, called singular lines, which we define below, which are not weakly attracted to $\Lambda^+_\phi$.

Define a \textit{singular line} for $\phi$ to be a line $\gamma\in \mathcal{B}$ if there exists a principal lift $\Phi$ of $\phi$ and a lift $\widetilde{\gamma}$ of $\gamma$ such that the endpoints of $\widetilde{\gamma}$ are contained in Fix$_N(\Phi) \subset \partial \F$.\\
The set of all singular lines of $\phi$ is denoted by $\mathcal{B}^{sing}(\phi)$. 
The fact [Lemma 2.1, \citep{HM-13c}] below summarizes this discussion.
\begin{lemma}
 Given a rotationless $\phi\in \out$ and an attracting lamination $\Lambda^+_\phi$, any line $\gamma$ that satisfies one of the following three conditions is in $\mathcal{B}_{na}(\Lambda^+_\phi)$.
\begin{enumerate}
 \item $\gamma$ is carried by $\mathcal{A}_{na}\Lambda^\pm_\phi$
  \item $\gamma$ is a generic leaf of some attracting lamination for $\phi^{-1}$
\item $\gamma$ is in $\mathcal{B}^{sing}(\phi)$.
\end{enumerate}
 
\end{lemma}

But these are not all lines that constitute $\mathcal{B}_{na}(\Lambda^+_\phi)$. A important theorem in [Theorem 2.6, \citep{HM-13c}, stated here as Lemma \ref{NAL}, tells us that there is way to concatenate lines from the three classes we mentioned in the above fact which will also result in lines that are not weakly attracted to $\Lambda^+_\phi$. Fortunately, these are all possible types of lines in 
$\mathcal{B}_{na}(\Lambda^+_\phi)$. A simple explanation of why the concatenation is necessary is, one can construct a line by connecting the base points of two rays, one of which is asymptotic to a singular ray in the forward direction of $\phi$ and the other is asymptotic to a singular ray in the backward direction of $\phi$. This line does not fall into any of the three categories we see in the fact above. 
The concatenation process described in \citep{HM-13c} takes care of such lines. We will not describe the concatenation here, but the reader can look up section 2.2 in \citep{HM-13c}. The following definition is by Handel and Mosher:

\begin{definition}

Let $A \in \mathcal{A}_{na}\Lambda^\pm_\phi$ and $\Phi \in P(\phi)$, we say that $\Phi$ is $A-related$ if Fix$_N(\widehat{\Phi})\cap \partial A \neq \emptyset$. Define the extended boundary of $A$ to be $$\partial^{ext}(A,\phi) = \partial A \cup \bigg( \bigcup_{\Phi}Fix_N(\widehat{\Phi}) \bigg)$$
where the union is taken over all $A$-related $\Phi\in P(\phi)$.
\end{definition}
Let $\mathcal{B}^{ext}(A,\phi)$ denote the set of lines which have end points in $\partial^{ext}(A,\phi)$; this set is independent of the choice of $A$ in its conjugacy class. Define $$\mathcal{B}^{ext}(\Lambda^+_\phi)  = \bigcup_{A\in \mathcal{A}_{na}\Lambda^\pm_\phi} \mathcal{B}^{ext}(A,\phi)$$. We can now state the main result about non-attracted lines.
\begin{lemma}
 \label{NAL}
Suppose $\phi, \psi = \phi^{-1}\in \out$ be rotationless elements and $\Lambda^+_\phi$ is an attracting lamination for $\phi$. Then any line $\gamma$ is in $\mathcal{B}_{na}(\Lambda^+_\phi)$ if and only if one of the following conditions hold:
\begin{enumerate}
 \item $\gamma$ is in $\mathcal{B}^{ext}(\Lambda^+_\phi)$
\item $\gamma$ is in $\mathcal{B}^{sing}(\phi)$
\item $\gamma$ is a generic leaf of some attracting lamination for $\psi$
\end{enumerate}
 
\end{lemma}

It is worth noting that the sets of lines mentioned in Lemma \ref{NAL} are not necessarily pairwise disjoint. But if we have a line $\sigma \in \mathcal{B}_{na}(\Lambda^+_\phi)$ that is bi-recurrent then the situation is much simpler.
In that case $\sigma$ is either a generic leaf of some attracting lamination for $\phi^{-1}$ or $\sigma$ is carried by $\mathcal{A}_{na}\Lambda^\pm_\phi$. This takes us to the weak attraction theorem that we need to prove our result.

\subsection{Weak attraction theorem }
\label{sec:9}
\begin{lemma}[\citep{HM-13c} Corollary 2.17]
\label{WAT}
 Let $\phi\in \out$ be a rotationless and exponentially growing. Let $\Lambda^\pm_\phi$ be a dual lamination pair for $\phi$. Then for any line $\gamma\in\mathcal{B}$ not carried by $\mathcal{A}_{na}(\Lambda^{\pm}_\phi)$ at least one of the following hold:
\begin{enumerate}
 \item $\gamma$ is attracted to $\Lambda^+_\phi$ under iterations of $\phi$.
   \item $\gamma$ is attracted to $\Lambda^-_\phi$ under iterations of $\phi^{-1}$.
\end{enumerate}
Moreover, if $V^+_\phi$ and $V^-_\phi$ are attracting neighborhoods for the laminations $\Lambda^+_\phi$ and $\Lambda^-_\phi$ respectively, there exists an integer $l\geq0$ such that at least one of the following holds:
\begin{itemize}
 \item $\gamma\in V^-_\phi$.
\item $\phi^l(\gamma)\in V^+_\phi$
\item $\gamma$ is carried by $\mathcal{A}_{na}(\Lambda^{\pm}_\phi)$. 
\end{itemize}

\end{lemma}

\begin{corollary}
\label{WATgeo}
 Let $\phi\in \out$ be exponentially growing and $\Lambda^{\pm}_\phi$ be geometric dual lamination pair for $\phi$ such that $\phi$ fixes $\Lambda^+_{\phi}$ and $\phi^{-1}$ fixes $\Lambda^{-}_\phi $with attracting neighborhoods $V^{\pm}_\phi$. 
 Then there exists some integer $l$ such that for any line $\gamma$ in $\mathcal{B}$ one of the following occurs:
\begin{itemize}
 \item $\gamma\in V^-_\phi$.
\item $\phi^l(\gamma)\in V^+_\phi $.
\item$\gamma$ is carried by $\mathcal{A}_{na}(\Lambda^{\pm}_\phi)$ 
\end{itemize}

\end{corollary}
\begin{proof}
 Let $K$ be a positive integer such that $\phi^K$ is rotationless. Then by definition $\mathcal{A}_{na}(\Lambda^{\pm}_\phi)=\mathcal{A}_{na}(\Lambda^{\pm}_{\phi^K})$. Also $\phi$ fixes $\Lambda^+_{\phi}$  implies 
$\Lambda^{+}_\phi=\Lambda^+_{\phi^K}$ and the attracting neighborhoods $V^+_\phi$ and $V^+_{\phi^K}$ can also be chosen to be the same weak neighborhoods.
Then by Lemma \ref{WAT} we know that there exists some positive integer $m$ such that the conclusions of the Weak attraction theorem hold for $\phi^K$. Let $l:=mK$. This gives us the conclusions of the corollary.
Before we end we note that by definition of an attracting neighborhood $\phi(V^+_\phi)\subset V^+_\phi$ which implies that if $\phi^l(\gamma)\in V^+_\phi$, then $\phi^t(\gamma)\in V^+_\phi$ for all $t\geq l$.
\end{proof}

\begin{lemma}
 Suppose $\phi,\psi\in \out$ are two exponentially growing automorphisms with attracting laminations $\Lambda^+_\phi$ and $\Lambda^+_\psi$, respectively. If a generic leaf $\lambda\in\Lambda^+_\phi$ is in $\mathcal{B}_{na}(\Lambda^+_\psi)$ then the whole lamination $\Lambda^+_\phi\subset \mathcal{B}_{na}(\Lambda^+_\psi)$. 
 
\end{lemma}

\begin{proof}
 Recall that a generic leaf is bi-recurrent. Hence, $\lambda\in \mathcal{B}_{na}(\Lambda^+_\psi)$ implies that $\lambda$ is either carried by $\mathcal{A}_{na}$ or it is a generic leaf of some element of $\mathcal{L}(\psi^{-1})$. 
 First assume that $\lambda$ is carried by $\mathcal{A}_{na}$. Then using Lemma \ref{NAS} item 4, we can conclude that $\Lambda^+_\phi$ is carried by $\mathcal{A}_{na}(\Lambda^+_\psi)$.
 
 Alternatively, if $\lambda$ is a generic leaf of some element $\Lambda^-_\psi\in \mathcal{L}(\psi^{-1})$, then the weak closure $\overline{\lambda} = \Lambda^+_\phi = \Lambda^-_\phi$ and we know $\Lambda^-_\psi$ does not get attracted to $\Lambda^+_\psi$. 
 Hence, $\Lambda^+_\phi\subset\mathcal{B}_{na}(\Lambda^+_\psi)$.
\end{proof}

\section{Pingpong argument for exponential growth}
\label{sec:10}
\begin{lemma}[\citep{BFH-00} Section 2.3]
\label{doublesharp}
 If $f:G\longrightarrow G$ is a train track map for an irreducible $\phi\in$Out($F_n$) and $\alpha$ is a path in some leaf $\lambda$ of $G$ such that $\alpha=\alpha_1\alpha_2\alpha_3$ is a decomposition into subpaths such that $|\alpha_1|,|\alpha_3|\geq2C $
where C is the bounded cancellation constant for the map $f$, then $f^k_\#(\alpha_2)\subset f^k_{\#\#}(\alpha)$ for all $k\geq0$.
\end{lemma}

\begin{proof}
 Let $\alpha$ be any path with a decomposition $\alpha=\alpha_1\alpha_2\alpha_3$. Take lifts to universal cover of $G$. If $\tilde{\gamma}$ is a path in $\tilde{G}$ that contains $\tilde{\alpha}$, then decompose $\tilde{\gamma}=\tilde{\gamma_1}\tilde{\alpha_2}\tilde{\gamma_3}$
such that $\tilde{\alpha_1}$ is the terminal subpath of $\tilde{\gamma_1}$ and $\tilde{\alpha_3}$ is the initial subpath of $\tilde{\gamma_3}$. 
 Following the proof of \citep{BFH-00} if $K=2C$ then $\tilde{\gamma}$ can be split at the endpoints of $\tilde{\alpha_2}$. Thus, $\tilde{f^k_\#}(\tilde{\gamma})=\tilde{f^k_\#}(\tilde{\gamma_1})\tilde{f^k_\#}(\tilde{\alpha_2})\tilde{f^k_\#}(\tilde{\gamma_3})$.
The result now follows from the definition of $f^k_{\#\#}(\alpha)$. 

\end{proof}
\begin{lemma}[\citep{HM-13c} Lemma 1.1]
 \label{expgrowth}
Let $f:G\rightarrow G$ be a homotopy equivalence representing $\phi\in \out$ such that there exists a finite path $\beta\subset G$ having the property that 
$f_{\#\#}(\beta)$ contains three disjoint copies of $\beta$. Then $\phi$ is exponentially growing and there exists a lamination $\Lambda\in\mathcal{L}(\phi)$ and a generic leaf $\lambda$ of $\Lambda\in\mathcal{L}(\phi)$ 
such that $\Lambda$ is $\phi$-invariant and $\phi$ fixes $\lambda$ preserving orientation, each generic leaf contains $f^i_{\#\#}(\beta)$ as a subpath 
for all $i\geq0$ and $N(G,\beta)$ is an attracting neighborhood for $\Lambda$.
\end{lemma}

We adapt the following notation for the statement of the next proposition: 
Let $\epsilon \in \{-,+\}$ and $i\in \{0,1\}$. Here $i$ will be used to represent $\psi$(if $i=0$) or $\phi$(if $i=1$). Together 
the tuple $\mu_i := (i,\epsilon) \in \{0,1\}\times \{-,+\}$ will represent $\psi$(if $\mu_0=(0,+)$), $\psi^{-1}$ (if $\mu_0=(0,-)$) and so on. In this notation we write $\Lambda_{0}^{\epsilon} := \Lambda^{\epsilon}_{\psi}$ and so on.\\
We also write notations like $-\epsilon$ where it means $--=+$ and $-+=-$ depending on value of epsilon.\\
\textbf{Standing assumptions for the rest of this section:} $\phi,\psi$ are exponentially growing elements of $\out$ such that the following conditions are satisfied:
\begin{enumerate}
  \item $\phi, \psi$ are not powers of one another.
 \item There exists dual lamination pairs $\Lambda^\pm_{\psi}$ and $\Lambda^\pm_{\phi}$ such that $\Lambda^\pm_{\psi}$ is attracted to $\Lambda^\epsilon_{\phi}$ under iterates of $\phi^\epsilon$ and $\Lambda^\pm_\phi$ is attracted to $\Lambda^\epsilon_\psi$ under 
 iterates of $\psi^\epsilon$.
 \item $\psi^\epsilon$ fixes $\Lambda^\epsilon_\psi$ and $\phi^\epsilon$ fixes $\Lambda^\epsilon_\phi$.
 \item Both $\Lambda^\pm_\psi$ and $\Lambda^\pm_\phi$ are non-geometric or every lamination pair of every element of $\langle \psi, \phi \rangle$ is geometric.
 
 \end{enumerate}

\begin{remark}
 Some remarks regarding the set of hypothesis.
 \begin{itemize}
  \item hypothesis 2 and 3 are needed to play the ping-pong game.
  \item hypothesis 4 is needed to prove that the attracting and repelling laminations produced out of ping-pong are dual. This, as we will see later in Proposition \ref{duality} , is not required if 
$\phi, \psi$ are fully irreducible.
   
 \end{itemize}

\end{remark}

The lemma that has been proven by Handel and Mosher in Proposition 1.3 in \citep{HM-13c} is a special case of the following proposition. What they have shown (with slightly weaker conditions than hypothesis 2 above) is that the lemma is true for $k=1$ and only under positive 
powers of $\psi$ and $\phi$. Strengthening that one part of the hypothesis enables us to extend their result to both positive and negative exponents and also for reduced words with arbitrary $k$ (see statement of \ref{pingpong} for description of $k$). They also have the assumption that $\phi,\psi$ are both 
rotationless, which they later on discovered, is not necessary; one can get away with the hypothesis 3 above. The main technique, however, is same.

\textbf{The Handel-Mosher pingpong argument:}

\begin{proposition}
\label{pingpong}
 Let $\psi,\phi$ be exponentially growing elements of $\out$, that satisfy the hypothesis mentioned above.

Then there exists some integer $M> 0$ and attracting neighborhoods $V^\pm_\phi$ , $V^\pm_\psi$ of $\Lambda^\pm_\phi$ and $\Lambda^\pm_\psi$ ,respectively, such that for every pair of finite sequences $n_i\geq M$ and $m_i\geq M$ if

 $$\xi= \psi^{\epsilon_1 m_1}\phi^{\epsilon_1 ' n_1}........\psi^{\epsilon_k m_k}\phi^{\epsilon_k ' n_k}$$ 
$(k\geq 1)$ is a cyclically reduced word then $w$ will be exponentially-growing and have a lamination pair $\Lambda^\pm_\xi$ satisfying the following properties:
\begin{enumerate}
  \item Every conjugacy class carried by $\mathcal{A}_{na}(\Lambda^\pm_\xi)$ is carried by both $\mathcal{A}_{na}(\Lambda^\pm_\phi)$ and $\mathcal{A}_{na}(\Lambda^\pm_\psi)$
  \item $\psi^{m_i}(V^\pm_\phi) \subset V^+_\psi$ and $\psi^{-m_i}(V^\pm_\phi) \subset V^-_\psi$ .
  \item $\phi^{n_j}(V^\pm_\psi) \subset V^-_\phi$ and $\phi^{-n_j}(V^\pm_\psi) \subset V^-_\phi$ .
  \item $V^+_\xi \colon= V^{\epsilon_1}_\psi$ is an attracting neighborhood of $\Lambda^+_\xi$ 
  \item $V^-_\xi \colon= V^{-\epsilon_k '}_\phi$ is an attracting neighborhood of $\Lambda^-_\xi$ 
  \item (uniformity) Suppose $U^{\epsilon_1}_\psi$ is an attracting neighborhood of $\Lambda^{\epsilon_1}_\psi$ and $\lambda^+_\xi$ is a generic leaf of $\Lambda^+_\xi$. Then $\lambda^+_\xi\in U^{\epsilon_1}_\psi$ for 
  sufficiently large $M$.
  \item (uniformity) Suppose $U^{\epsilon_k'}_\phi$ is an attracting neighborhood of $\Lambda^{\epsilon_k'}_\phi$ and $\lambda^-_\xi$ is a generic leaf of $\Lambda^-_\xi$. Then $\lambda^-_\xi\in U^{\epsilon_k'}_\phi$ for 
  sufficiently large $M$.
\end{enumerate}

\end{proposition}

\begin{proof}

Let $g_{\mu_i}:G_{\mu_i}\longrightarrow G_{\mu_i}$ be stable relative train train-trak maps and \\ $u^{\mu_i}_{\mu_j}:G_{\mu_i}\longrightarrow G_{\mu_j}$ be the homotopy equivalence between the graphs which preserve the markings, where $i\neq j$.

Let $C_1> 2\text{BCC}\{g_{\mu_i}|i \in\{0,1\}\}$. Let $C_2> \text{BCC}\{u^{\mu_i}_{\mu_j}| i,j\in\{0,1\},i\neq j\}$. Let
 $C\geq C_1,C_2$.

Let $\lambda_i^{\epsilon}$ be generic leaves of laminations $\Lambda_i^{\epsilon}$.
\paragraph{Step 1:}
Using the fact that $\Lambda_1^{\epsilon}$ is weakly attracted to $\Lambda_0^{\epsilon '}$, under the action if $\psi^{\epsilon '}$, choose a finite subpath $\alpha_1^{\epsilon}\subset\lambda^{\epsilon}_1$ such that 
\begin{itemize}
 \item $(u_{\mu_0}^{\mu_1})_{\#}(\alpha_1^{\epsilon})\rightarrow \lambda^{\epsilon '}_0$ weakly, where $\mu_0=(0,\epsilon ')$ and $\mu_1=(1,\epsilon)$.
\item $\alpha_1^{\epsilon}$ can be broken into three segments: initial segment of $C$ edges, followed by a subpath $\alpha^{\epsilon}$ followed by a terminal segment with $C$ edges.
\end{itemize}

\paragraph{Step 2:}
 Now using the fact $\Lambda^{\epsilon}_0\rightarrow \Lambda^{\epsilon '}_1$ weakly, under iterations of $\phi^{\epsilon '}$, we can find positive integers $p^{\epsilon}_{\mu_1}$ (there are four choices here that will yield four integers) such that 
$\alpha^{\epsilon'}_1 \subset (g^{p^{\epsilon}_{\mu_1}}_{\mu_1} u_{\mu_1}^{\mu_0})_\#(\lambda^{\epsilon}_0)$ , where $\mu_0=(0,\epsilon) , \mu_1=(1,\epsilon')$. \\
Let $C_3$ be greater than BCC$\{g^{p^{\epsilon}_{\mu_1}}_{\mu_1} u_{\mu_1}^{\mu_0}\} $ (four maps for four integers $p^{\epsilon}_{\mu_1}$) .\\

\paragraph{Step 3:}
Next, let $\beta^{\epsilon}_0\subset \lambda^{\epsilon}_0$ be a finite subpath such that $(g^{p^{\epsilon}_{\mu_1}}_{\mu_1})_\#(\beta^{\epsilon}_0)$ contains $\alpha^{\epsilon '}_1$ protected by $C_3$ edges in both sides, where $\mu_0=(0,\epsilon)$ and $\mu_1=(1,\epsilon ')$.
Also, by increasing $\beta^{\epsilon}_0$ if necessary, we can assume that $V^{\epsilon}_\psi = N(G_{\mu_0},\beta^{\epsilon}_0)$ is an attracting neighborhood of $\Lambda^{\epsilon}_0$ .

Let $\sigma$ be any path containing $\beta^{\epsilon}_0$. Then $(g^{p^{\epsilon}_{\mu_1}}_{\mu_1} u_{\mu_1}^{\mu_0})_\#(\sigma)\supset \alpha^{\epsilon}_0$. Thus by using Lemma \ref{doublesharp} we get that  $(g^{p^{\epsilon}_{\mu_1}+t}_{\mu_1} u_{\mu_0}^{\mu_1})_\#(\sigma)=(g^t_{\mu_1})_\#((g^{p^{\epsilon}_{\mu_1}}_{\mu_1} u_{\mu_1}^{\mu_0})_\#(\sigma))$ contains $(g^t_{\mu_1})_\#(\alpha^{\epsilon})$ for all $t\geq 0$ .\\
Thus we have $(g^{p^{\epsilon}_{\mu_1}+t}_{\mu_1} u_{\mu_1}^{\mu_0})_{\#\#}(\beta^{\epsilon}_0)\supset (g^t_{\mu_1})_\#(\alpha^{\epsilon})$ for all $t\geq0$.

\paragraph{Step 4:}
Next step is reverse the roles of $\phi$ and $\psi$ to obtain positive integers $q^{\epsilon'}_{\mu_1}$ and paths $\gamma^{\epsilon'}_1\subset\lambda^{\epsilon'}_{\mu_1}$ such that $(g^{q^{\epsilon}_{\mu_0}+t}_{\mu_0} u_{\mu_0}^{\mu_1})_{\#\#}(\gamma^{\epsilon'}_1)\supset (g^t_{\mu_0})_\#(\beta^{\epsilon}_0)$ for all $t\geq0$ 
, where $\mu_0=(0,\epsilon), \mu_1= (1,\epsilon')$  

\paragraph{Step 5:}
Finally, let $k$ be such that $(g^k_{\mu_1})_{\#}(\alpha^{\epsilon})$ contains three disjoint copies of $\gamma^{\epsilon}_1$ and that $(g^{k}_{\mu_0})_{\#}(\beta^{\epsilon}_0)$ contains three 
disjoint copies of $\beta^{\epsilon}_0$ for $\epsilon = 0,1$. Let $p\geq$ max $\{ p^{\epsilon}_{\mu_1} \}$  + $k$ and $q\geq$ max $\{q^{\epsilon}_{\mu_0} \}$  + $k$.

Let $m_i\geq q$ and $n_i\geq p$.

The map $f_\xi = g^{m_1}_{(0,\epsilon_1)}u^{(1,\epsilon_1')}_{(0,\epsilon_1)}g^{n_1}_{(1,\epsilon_1')}u^{(0,\epsilon_2)}_{(1,\epsilon_1')}.........g^{n_k}_{(1,\epsilon_k ')}u^{(0,\epsilon_1)}_{(1,\epsilon_k ')}:G_{(0,\epsilon_1)}\rightarrow G_{(0,\epsilon_1)}$ is a topological representative of $\xi$.
With the choices we have made, \\ $g^{n_k}_{(1,\epsilon_k')}u^{(0,\epsilon_1)}_{(1,\epsilon_k')})_{\#\#}(\beta^{\epsilon_1}_0)$ contains three disjoint copies of $\gamma_1^{\epsilon_k'}$ and so \\ $(g^{m_k}_{(0,\epsilon_k)}u^{(1,\epsilon_k')}_{(0,\epsilon_k)}g^{n_k}_{(1,\epsilon_k')}u^{(0,\epsilon_1)}_{(1,\epsilon_k')})_{\#\#}(\beta^{\epsilon_1}_0) $ will contain three disjoint copies of $\beta_0^{\epsilon_k}$. 
Continuing in this fashion in the end we get that $(f_\xi)_{\#\#}(\beta_0^{\epsilon_1})$ contains three disjoint copies of $\beta_0^{\epsilon_1}$. Thus by Lemma \ref{expgrowth} $\xi$ is an exponentially growing element of $\out$ with an attracting lamination $\Lambda^+_\xi$ which has 
$V^+_\xi = \\ N(G_{\mu_0}, \beta^{\epsilon_1}_0) = V^{\epsilon_1}_\psi$ as an attracting neighborhood. \\

Similarly, if we take inverse of $\xi$ and interchange the roles played by $\psi, \phi$ with $\phi^{-1},\psi^{-1}$, we can produce an attracting lamination $\Lambda^-_{\xi}$ for $\xi^{-1}$ with an attracting neighborhood $V^-_\xi = N(G_{(1,-\epsilon_k')},\gamma_1^{-\epsilon_k'})=V^{-\epsilon_k'}_\phi$ . 
 which proves  property (4) and (5) of the proposition.
We shall later show that $\Lambda^+_\xi$ and $\Lambda^-_\xi$ form a dual-lamination pair.

Hence,every reduced word of the group $\langle\phi^n,\psi^m\rangle$ will be exponentially growing if $n\geq p,m\geq q$. Let $M\geq p,q$. 

Now, we prove the conclusion related to non-attracting subgroup.
By corollary \ref{WATgeo}  there exists $l$ so that if $\tau$ is neither an element of $V^{-\epsilon_k'}_\phi = V^-_\xi$ nor is it carried by $\mathcal{A}_{na}\Lambda^\pm_\phi$ then $\phi^{\epsilon_k't}_\#(\tau)\in V^{\epsilon_k'}_\phi$ for all $t\geq l$ . Increase $M$
 if necessary so that $M>l$. Under this assumption, 

$\xi_\#(\tau)\in \psi^{\epsilon_1m_1}(V^{\epsilon_1'}_\phi)\subset V^+_\xi$. So $\tau$ is weakly attracted to $\Lambda^+_\xi$. Hence we conclude that if $\tau\notin V^-_\xi$ and not attracted to $\Lambda_\xi^+$, then $\tau$ is  carried by $\mathcal{A}_{na}\Lambda^\pm_\phi$ .

Similarly, if $\tau$ is not in $V^+_\xi$ and not attracted to $\Lambda^-_\xi$ then $\tau$ is carried by $\mathcal{A}_{na}\Lambda^\pm_\psi$.

Next, suppose that $\tau$ is a line that is not attracted to any of $\Lambda^+_\xi, \Lambda^-_\xi$. Then $\tau$ must be disjoint from $V^+_\xi, V^-_\xi$. So, is carried by both $\mathcal{A}_{na}\Lambda^\pm_\psi$ and $\mathcal{A}_{na}\Lambda^\pm_\phi$.
Restricting our attention to periodic line, we can say that every conjugacy class that is carried by both $\mathcal{A}_{na}\Lambda^+_\xi$ and $\mathcal{A}_{na}\Lambda^-_\xi$ is carried by both $\mathcal{A}_{na}\Lambda^\pm_\psi$ and $\mathcal{A}_{na}\Lambda^\pm_\phi$. 
That the two non-attracting subgroup systems are mutually malnormal gives us the first conclusion.

The proof in  \citep{HM-13c} that $\Lambda^-_\xi$ and $\Lambda^+_\xi$ are dual lamination pairs will carry over in this situation and so $\mathcal{A}_{na}\Lambda^+_\xi=\mathcal{A}_{na}\Lambda^-_\xi$ .

Thus we have the proof of the first conclusion of our proposition.
\end{proof}

The following are the main ingredients that will be used in the to show applications of the main theorem of this paper:
\begin{proposition}
 \label{duality}
 If we assume that $\psi,\phi$ are fully-irreducible, then we do not need the second and third  bulleted item in the hypothesis for the ping-pong proposition.
\end{proposition}

\begin{remark} In the case of fully-irreducible, situation is much simpler
 \begin{enumerate}
  \item Hypothesis 1 implies hypothesis 2.
  \item Hypothesis 3 is obvious since the attracting and repelling lamination pairs for fully irreducible elements are unique.
  \item Hypothesis 4 is needed to prove that the laminations produced from the ping-pong argument are dual. We will see a direct proof in the line of Proposition 1.3 \citep{HM-13c} if the elements are fully irreducible, without using hypothesis 3 needed for propostion \ref{pingpong}
 \end{enumerate}

\end{remark}

\begin{proof}
 We will show that the laminations produced from ping-pong argument are dual. As in the proposition assume $$\xi= \psi^{\epsilon_1 m_1}\phi^{\epsilon_1 ' n_1}........\psi^{\epsilon_k m_k}\phi^{\epsilon_k ' n_k} $$
 Suppose the laminations $\Lambda^+_\xi$ and $\Lambda^-_\xi$ produced using the ping-pong type argument are not dual. Index all the dual lamination pairs of $\xi$ as $\{\Lambda^\pm_{i}\}_{i\in I}$ and assume that 
 \[\Lambda^+_\xi=\Lambda^+_i \,\, \mathrm{ and } \,\, \,  \Lambda^-_{\xi} = \Lambda^-_j \, \textrm{ for some } i\neq j  \]
 \textbf{Case 1: } $\Lambda^+_i\nsubseteq \Lambda^+_j$. This implies that a generic leaf $\lambda$ of $\Lambda^+_j$ is not attracted to $\Lambda^+_i=\Lambda^+_\xi$ under iteration by $\xi$.
 Also, $\lambda$ is not attracted to $\Lambda^-_j = \Lambda^-_\xi$ under iteration by $\xi^{-1}$. In particular, $\lambda\notin V^-_\xi$. By the discussion at the end of the ping-pong argument, this implies that $\lambda$ is carried by $\mathcal{A}_{na}\Lambda^\pm_\phi$. 
 But $\phi$ being fully-irreducible, $\mathcal{A}_{na}\Lambda^\pm_\phi$ is either trivial, which gives us that $\lambda$ does not exist(contradiction), or $\mathcal{A}_{na}\Lambda^\pm_\phi= [c]$ , which implies that $\lambda$ is a circuit. The later is impossible since a generic leaf cannot be a circuit (Lemma 3.1.16, \citep{BFH-00}).\\
 \textbf{Case 2: } $\Lambda^+_i\subset \Lambda^+_j$. By Lemma \ref{lam_incl} this means $\Lambda^-_i\subset \Lambda^-_j$. This implies that some generic leaf $\lambda$ of $\Lambda^-_i$ is not attracted to $\Lambda^-_j$ under iteration by $\xi^{-1}$ (since proper inclusion implies there is a generic leaf of $\Lambda_i$ whose height is less than the height of the EG stratum corresponding to $\Lambda_j$). 
 Also, $\lambda$ is not attracted to $\Lambda^+_i = \Lambda^+_\xi$ under iteration by $\xi$. In particular, $\lambda\notin V^+_\xi$.
 By discussion at the end of  the ping-pong proposition, this implies that $\lambda$ is carried by $\mathcal{A}_{na}\Lambda^\pm_\psi$. The same arguments as in case 1 works and we get a contradiction.
 
\end{proof}

\begin{corollary}
\label{hyp1}
 If in proposition \ref{pingpong} if we drop bulleted items 2, 3 in the hypothesis and instead assume that $\psi,\phi$ are fully-irreducible outer automorphisms such that $\phi$ is geometric and $\psi$ is hyperbolic(or vice versa), then the 
 resulting laminations $\Lambda^+_\xi$ and $\Lambda^-_\xi$ produced by the ping-pong argument will be dual. Moreover, $\mathcal{A}_{na}\Lambda^\pm_\xi$ will be trivial if $\xi$ is not conjugate to a power of $\phi$
\end{corollary}
\begin{proof}
 The result follows from Proposition \ref{duality} and the conclusion 1 from Proposition \ref{pingpong}
\end{proof}

\begin{corollary}
\label{hyp2}
 If in proposition \ref{pingpong} if we drop bulleted items 2, 3 in the hypothesis and instead assume that $\psi$ and $\phi$ are fully-irreducible and geometric and fix the same conjugacy class, then the resulting laminations $\Lambda^+_\xi$ and $\Lambda^-_\xi$ produced by the ping-pong argument will be dual
  and $\Lambda^\pm_\xi$ will be geometric and $\mathcal{A}_{na}(\Lambda^\pm_\xi)$ will be equal to $\mathcal{A}_{na}(\Lambda^\pm_\phi)$. If they don't fix the same conjugacy class, then $\mathcal{A}_{na}\Lambda^\pm_\xi$ is trivial when $\xi$ is not conjugate to a power of $\phi$ or $\psi$
\end{corollary}

\begin{proof}
 When both are geometric and fix the same conjugacy class they arise from pseudo-Anosov homeomorphism of the same surface with connected boundary and the conjugacy class corresponding to the boundary, $[c]$ is equal to $\mathcal{A}_{na}\Lambda^\pm_\phi$ and $\mathcal{A}_{na}\Lambda^\pm_\psi$.
 So, every reduced word $\xi$ in $\phi$ and $\psi$ will fix $[c]$. We get the conclusion about duality of $\Lambda^+_\xi$ and $\Lambda^-_\xi$ by using proposition \ref{duality} and conclusion 1 of proposition \ref{pingpong} tells us that $\mathcal{A}_{na}\Lambda^\pm_\xi = [c]$. \\
If they are both geometric but they do not fix the same conjugacy class $\mathcal{A}_{na}(\Lambda^\pm_\phi)$ and $\mathcal{A}_{na}(\Lambda^\pm_\psi)$ are conjugacy classes of infinite cyclic subgroups which have generators which are not powers of each other.
Proposition \ref{duality} tells us that the laminations $\Lambda^+_\xi$ and $\Lambda^-_\xi$ are dual. Then using the conclusion(1) from pingpong Lemma \ref{pingpong} we can conclude that $\mathcal{A}_{na}(\Lambda^\pm_\xi)$ does not carry any conjugacy classes and hence is trivial when $\xi$ is not conjugate to some power of $\psi$ or $\phi$.
\end{proof}




\section{Proof of main theorem}
\label{sec:11}
We begin this section by introducing the concept of Stallings graph associated to a free factor, which will contain the information about lines which are carried by the free factor. Stallings introduced the concept of
folding graphs in his seminal work in \citep{St-83}

\textbf{Stallings graph :}  Consider a triple $(\Gamma, S, p)$ consisting of a marked graph $\Gamma$, a connected subgraph $S$ with no valence one vertices and a homotopy equivalence $p : \Gamma\rightarrow R_N$, which is a homotopy inverse of the marking on $\Gamma$. 
$p$ takes vertices to vertices and edges to edge-paths and is a immersion when restricted to $S$. This enforces that any path in $S$ is mapped to a path in $R_N$. If in addition we have $[F]=[S]$, we say that the triple $(\Gamma, S, p)$ is a representative of the free factor $F$ and $S$ is
 said to be the \textit{Stallings graph} for $F$. We get a metric on each edge of $\Gamma$ by pulling back the metric from $R_N$ via $p$. Under this setting, a line $l$ is carried by $[F]$ if and only if its realization $l_\Gamma$ in $\Gamma$ is contained in $S$, in which case the restriction of 
$p$ to the line $l_\Gamma$ is an immersion whose image in $R_N$ is $l$.

\begin{lemma}
 Every proper free factor $F$ has a realization $(\Gamma, S, p)$.
\end{lemma}
 The description of the Stallings graph and the proof of the fact mentioned above can be found in the proof of Theorem I, section 2.4 in \citep{HM-13c}. The proof of existence uses the Stallings fold theorem to construct $S$, hence the name.

The following fact is an important tool for the proof of the main theorem. It is used to detect fully-irreducibility when we are given an exponentially growing element.
\begin{lemma}
\label{fih}
 Let $\phi\in \out$ be a rotationless and exponentially growing element. Then for each attracting lamination $\Lambda^+_\phi$, if the subgroup system $\mathcal{A}_{na}\Lambda^+_\phi$ is trivial and the free factor system $\mathcal{A}_{supp}(\Lambda^+_\phi)$ is not proper then $\phi$ is fully irreducible.
\end{lemma}

It is worth noting that the fully irreducible element we will get from this lemma is hyperbolic, since $\mathcal{A}_{na}\Lambda^+$ trivial implies that there are no periodic conjugacy classes and by \citep{Br-00} it is hyperbolic. In the next lemma we will extend this result to include the geometric case also.

\begin{lemma}
\label{mcg}
 For each exponentially growing $\phi\in \out$ if there exists an attracting lamination $\Lambda^+_\phi$ such that $\mathcal{A}_{na}\Lambda^+_\phi = [\langle c \rangle]$, where $\mathcal{A}_{supp} [c]$ and $\mathcal{A}_{supp}(\Lambda^+)$ are not proper then $\phi$ is fully irreducible and  geometric.
\end{lemma}

\begin{proof}
 We follow the footsteps of the proof of the fact \ref{fih}. Suppose $\phi$ is not fully irreducible. Pass on to a rotationless power and assume $\phi$ is rotationless. Let $[F]$ be the conjugacy class of the proper, non trivial  free factor fixed by $\phi$. 
 Choose a CT $f:G \rightarrow G$ such that $F$ is realized by some filtration element $G_r$ and $[G_r] = [F]$. Since $\mathcal{A}_{supp}(\Lambda^+)$ is not proper, the lamination $\Lambda^+_\phi$ corresponds to the highest strata $H_s$ and $r<s$. Next recall that a strata $H_i\subset G\setminus Z$ if and only if 
  there exists some $k \geq 0$ so that some term in the complete splitting of $f^k_\# (E_i)$ (for some edge $E_i \subset H_i$)is an edge in $H_s$. This implies that $G_{s-1} \subset Z$, since $f$ preserves the filtration. Hence we have $G_r\subset G_{s-1}\subset Z$. This implies that $[F] \lneq[\langle c \rangle]$. 
  So, $F = \langle c ^p \rangle$ for some $p>1$. But the conjugacy class $[c]$ fills contradicts our assumption that $F$ is a proper non trivial free factor.
\end{proof}

We write down the hypothesis for main theorem followed by some remarks as to how the assumptions are used.

\begin{definition}
\label{independent}

Let $\phi, \psi \in \out$ be exponentially growing outer automorphisms, which are not conjugate to powers of each other and which do not have a common periodic free factor. 
Also let $\phi,\psi$ have some dual lamination pairs $\Lambda^\pm_\phi$ and $\Lambda^\pm_\psi$, respectively, such that the following hold:
\begin{itemize}

 \item $\psi^\epsilon$ leaves $\Lambda^\epsilon_\psi$ invariant and $\phi^\epsilon$ leaves $\Lambda^\epsilon_\phi$ invariant.
 \item $\Lambda^\pm_{\psi}$ is attracted to $\Lambda^\epsilon_{\phi}$ under iterates of $\phi^\epsilon$ and $\Lambda^\pm_\phi$ is attracted to $\Lambda^\epsilon_\psi$ under 
 iterates of $\psi^\epsilon$
 \item $\{\Lambda^\pm_\phi\} \cup \{\Lambda^\pm_\psi\}$ fills.
 \item Both $\Lambda^\pm_\psi$ and $\Lambda^\pm_\phi$ are non-geometric or every lamination pair of every element of $\langle \psi, \phi \rangle$ is geometric.
 \item The non-attracting subgroup systems $\mathcal{A}_{na}(\Lambda^\pm_\phi)$ and $\mathcal{A}_{na}(\Lambda^\pm_\psi)$ are mutually malnormal.
\end{itemize}

 Any two pairs $(\phi, \Lambda^\pm_\phi), (\psi,\Lambda^\pm_\psi)$ which satisfy the criterions above will be called \textit{pairwise independent}. When we abuse the definition and say that  \textit{$\phi$, $\psi$ are pairwise independent}, it is understood that we also have 
 some dual lamination pairs $\Lambda^\pm_\phi$ and $\Lambda^\pm_\psi$ which satisfy the set of conditions above.
 
\end{definition}

\begin{remark}
Bulleted item 5 is a technical requirement to prove the first conclusion of the Proposition \ref{pingpong}. The rest are as follows:
 \begin{enumerate}
  \item The first two bullets are required to prove the conclusion about the free group of rank two in the statement of Theorem \ref{main} and are also needed to apply Proposition \ref{pingpong}.
  \item The third bullet will be used to deduce a contradiction in the proof of showing that $\xi$ is fully irreducible
 \item The fifth bullet implies that there is no line that is carried by both the nonattracting subgroup subgroup systems. This will be used to conclude hyperbolicity.
 \end{enumerate}

\end{remark}


The following result is same as Lemma 3.4.2 in \citep{BFH-00}.
\begin{lemma}
\label{free}
Let $\phi,\psi\in \out$ be exponentially growing elements, which are not conjugates to powers of each other, which satisfy the first three bullets . Then there exists some $M_0 \geq 0$ such that the group $G_{M_0} = \langle \psi^m , \phi^n \rangle$ is free for every $m,n \geq M_0$
\end{lemma}




We say that $l$ is a \textit{periodic line} if $l = .....\rho\rho\rho....$ is a bi-infinite iterate of some finite path $\rho$. In this case we write $l = \rho_\infty$

\begin{theorem}
\label{main}
 Let $\phi , \psi $ be two  exponentially growing elements of $\out$, such that there exists some dual lamination pairs $\Lambda^\pm_\phi$ and $\Lambda^\pm_\psi$, which makes $\phi$, $\psi$ pairwise independent. 
Then there exists an $M\geq0$, 
such that for all $n,m\geq M$ the group $\langle\psi^m,\phi^n\rangle$ will be free of rank two and every element of this free group, 
not conjugate (in $G$) to some power of the generators, will be hyperbolic and fully-irreducible.
\end{theorem}

\begin{proof}
 We already know that there exists some $M_0>0$, such that for all $m,n\geq M_0$ the group $ \langle\psi^m,\phi^n\rangle$ will be free group of rank two (Lemma\ref{free}).It remains to show that , 
 by increasing $M_0$ if necessary, every reduced word in this group, not conjugate to some power of the generators, will be fully irreducible. We shall prove it by contradiction.

Suppose that there does not exist any such $M\geq M_0$. This implies that for $M$ large, there exist $m(M), n(M)$, such that the group $\langle \psi^{m(M)}, \phi^{n(M)}\rangle$ contains at least one reduced word $\xi_M$ (not conjugate to some power of the generators) which is either reducible or fully irreducible but not hyperbolic. Using the hypothesis of mutual malnormality of $\mathcal{A}_{na}(\Lambda^\pm_\psi)$ and 
$\mathcal{A}_{na}(\Lambda^\pm_\phi)$ together with conclusion 1 of Proposition \ref{pingpong} we know that for sufficiently large $M$, $\mathcal{A}_{na}(\Lambda^\pm_{\xi_M})$ must be trivial. Hence after passing to a subsequence if necessary, assume that all $\xi_M$'s are reducible.

 We also make an assumption that this $\xi_M$ begins with a nonzero power of $\psi$ and ends in some nonzero power of $\phi$; if not, then we can conjugate to achieve this.
Thus as $M$ increases, we have a sequence of reducible elements $\xi_M\in \out$. Pass to a subsequence to assume that the $\xi_M$'s begin with a positive power of $\psi$ and end with a positive power of $\phi$. If no such subsequence exist, 
then change the generating set of the group by replacing generators with their inverses.

Let $\xi_M=\psi^{m_1}\phi^{\epsilon'_1 n_1}........\psi^{\epsilon_k m_k}\phi^{n_k}$ \hspace{2pc}where $m_i=m_i(M),n_j=n_j(M)$ and $k$ depend on $M$.

We note that by our assumptions, the exponents get larger as $M$ increases.
From the Ping-Pong lemma we know that there exists attracting neighborhoods $V^\pm_\psi$  and $V^\pm_\phi$ for  the dual lamination pairs $\Lambda^\pm_\psi$ and $\Lambda^\pm_\phi$ , respectively,  such that if $i\neq 1$ $$\psi^{{\epsilon_i m_i}(M)}(V^\pm_\phi)\subset V^{\epsilon_i}_\psi \textrm{ and } \psi^{m_1(M)}(V^\pm_\phi)\subset V^+_\psi\subset V^+_{\xi_M}$$
where each of $\xi_M$'s are exponentially-growing and equipped with a lamination pair $\Lambda^\pm_{\xi_M}$ (with attracting neighborhoods $V^\pm_{\xi_M}$) such that $\mathcal{A}_{na}\Lambda^\pm_{\xi_M}$ is trivial (using conclusion 1 of proposition \ref{pingpong} and bullet 6 in the hypothesis set)

Using Lemma \ref{fih} the automorphisms $\xi_M$'s being reducible implies that $\mathcal{A}_{supp}(\Lambda^\pm_{\xi_M})=[F_{\xi_M}]$ is proper. Fix a marked metric graph $H=R_N$, the standard rose. Denote the stallings graph (discussed at the beginning of this section) associated to $F_{\xi_M}$ by $K_M$, equipped with the immersion $p_M:K_M \rightarrow H$. A \textit{natural vertex} is a vertex with valence greater than two and a \textit{natural edge} is an edge between two natural vertices. We can subdivide every natural edge of $K_M$ into \textit{edgelets}, so that each edgelet is mapped to an edge in $H$ and label the edgelet by its image in $H$. 

Let $\gamma_M^-$ be a generic leaf of $\Lambda^-_M$ and $\gamma^+_M$ be a generic leaf of $\Lambda^+_M$.  We note that every natural edge in $K_M$ is mapped to an edge path in $H$, which is crossed by both $\gamma^-_M$ and $\gamma^+_M$. 

We claim that the edgelet length of every natural edge in $K_M$ is uniformly bounded above. \textit{Once we have proved the above claim}, it immediately follows that (after passing to a subsequence if necessary) there exists homeomorphisms $h_{M,M'}: K_M \rightarrow K_{M'}$ which maps edgelets to edgelets and preserves labels. Hence, we can assume that the sequence of graphs $K_M$ is eventually constant (upto homeomorphism) and $F_{\xi_M}=F$, is independent of $M$. 

Next, observe that, if $\alpha$  is any finite subpath of a generic leaf of $ \Lambda^+_\psi$, by enlarging $\alpha$ if necessary we can assume that it defines an attracting neighborhood of $\Lambda^+_\psi$.
By using the uniformity of attracting neighborhoods from the ping-pong lemma (conclusion 4,5) we know that $\gamma^+_M$ belongs to this neighborhood for sufficiently large $M$. This means $\alpha \subset \gamma^+_M$ for sufficiently large $M$, which implies that the realizations of $\lambda^+_\psi$ lift to $K_M$. A similar argument gives the same conclusion about $\lambda^-_\phi$. Thus both $\lambda^+_\psi$ and $\lambda^-_\phi$ are carried by $F$ , which implies that $F$ carries $\Lambda^+_\psi$ and $\Lambda^-_\phi$ which contradicts our hypothesis. Hence, $F$ cannot be proper and so the $\xi_M$'s are fully irreducible for all sufficiently large $M$ - contradiction. 

\textit{proof of claim :} Suppose that the edgelet length of the natural edges of $K_M$ is not uniformly bounded. Then there exists a  sequence of natural edges $\{E_M\}$ such that their edgelet lengths go to infinity as $M\rightarrow \infty$. Let $l$ be a weak limit of some subsequence $\{E_M\}$ and $\sigma\subset l$ be any finite subpath. For sufficiently large $M$, $\sigma\subset E_M \subset \gamma^+_M$ . Hence $l\in L^+ = \{$All weak limits of all subsequences of $\gamma^+_M\}$. Similarly, $l\in L^-=\{$All weak limits of all subsequences of $\gamma^-_M\}$. It remains to show that the intersection of this two sets is empty. Suppose not.

Let $\gamma^*$ be a weak limit of some subsequence of $\gamma^-_M$. We claim that $\gamma^*$ is not attracted to $\Lambda^+_\phi$. If not, then $\phi^p(\gamma^*)\in V^+_\phi$ for some $p\geq 0$. This means that for sufficiently large $M$, $\phi^p(\gamma^-_M)\in V^+_\phi$, implying that $\xi_M (\gamma^-_M)\in V^+_{\xi_M}$ for sufficiently large $M$, which is a contradiction to the fact that a generic leaf of $\Lambda^-_{\xi_M}$ is not attracted to $\Lambda^+_{\xi_M}$ under action of $\xi_M$. By similar arguments we can show that if $\gamma^*$ is a weak limit of some sequence of $\gamma^+_M$, then $\gamma^*$ is not attracted to $\Lambda^-_{\psi}$.

Let $l\in L^+ \cap L^-$. Then $l\in \mathcal{B}_{na}(\Lambda^+_\phi)\cap \mathcal{B}_{na}(\Lambda^-_\psi)$ by above arguments and by fact \ref{NAL}. If $l$ is not carried by $\mathcal{A}_{na}(\Lambda^\pm_\psi)$, then by the weak attraction theorem (Lemma \ref{WAT} ) $l$ is contained in every attracting neighborhood of the generic leaf $\lambda^+_\psi$.  This implies that $\lambda^+_\psi\subset cl(l) \subset \mathcal{B}_{na}(\Lambda^+_\phi) $. But this contradicts our hypothesis that $\Lambda^+_\psi$ is attracted to $\Lambda^+_\phi$. Hence $l$ must be carried by $\mathcal{A}_{na}(\Lambda^\pm_\psi)$.
By a symmetric argument we can show that $l$ must be carried by $\mathcal{A}_{na}(\Lambda^\pm_\phi)$. But this is not possible, since by our assumption these two subgroup systems are mutually malnormal. So $L^+ \cap L^- =\emptyset$

\end{proof}

\section{APPLICATIONS }
\label{sec:12}
This section is dedicated to looking at some applications of the Theorem \ref{main} . The corollary that is stated below is probably the first time that we see some result on the dynamics of mixed type of fully irreducible automorphisms. There are in fact three corollaries packed under the same hood. 
The first one is a well known theorem from \citep{KL-10} but the proof in their paper is very different from the technique we use here. The other two items are new.

\begin{corollary}
\label{app}
 If $\phi, \psi$ are fully-irreducible elements of $\out$ which are not conjugate to powers of each other, then there exists an integer $M\geq 0$ such that for every $m,n \geq M$, $G= \langle \psi^m, \phi^n \rangle$ is a free group of rank two, all whose elements are fully-irreducible.\\
 Moreover, $M$ can be chosen so that
 \begin{enumerate}
  \item If both $\phi, \psi$ are hyperbolic, then every element of $G$ is hyperbolic.
  \item If $\psi$ is hyperbolic and $\phi$ is geometric, then every element of $G$ not conjugate to a power of $\phi$ is hyperbolic.
  \item If both $\psi$ and $\phi$ are geometric but do not fix the same conjugacy class, then every element of $G$ not conjugate to a power of $\phi$ or $\psi$ is hyperbolic.
 \end{enumerate}

\end{corollary}

\begin{proof}
 The conclusion about the free group is from Proposition \ref{free}. It is easy to check that $\phi,\psi$ satisfy the hypothesis of Theorem \ref{main} except bullet four in definition\ref{independent}. But Proposition \ref{duality} tells us that bullet four is not required to draw all the conclusions in pingpong for this special case. Hence we can make the conclusion of the theorem which, along with the fact that the conjugate of any power of a 
fully-irreducible outer automorphism is also fully-irreducible, gives us that every element of $G$ is fully-irreducible.

 We now look at the proofs of the statements in moreover part:

 \begin{enumerate}
  \item follows immediately since conclusion 1 of Proposition\ref{pingpong} tells us that the $\mathcal{A}_{na}\Lambda^\pm_\xi$ is trivial which implies that no element of $G$ has any periodic conjugacy classes. \citep{Br-00} tells us that they are hyperbolic.
  \item follows from corollary \ref{hyp1} and Theorem \ref{main}
  \item follows from corollary \ref{hyp2} and Theorem \ref{main}
 \end{enumerate}

\end{proof}

The following result is another interesting application of the main theorem and several other technical lemmas that we have developed along the way. The proof of the result is almost same as Theorem \ref{main} same but a small modification is needed in the last part of the proof. 

A version of the theorem below, when the surface $\mathcal{S}$ is without boundary is an important result proved in \citep{FaM-02} where they develop the theory of convex cocompact subgroups of $\mathcal{MCG}(\mathcal{S})$.

\begin{theorem}
\label{anosov}
Let $\mathcal{S}$ be a connected, compact surface (not necessarily oriented) with one boundary component. Let $f,g \in \mathcal{MCG(S)}$ be pseudo-Anosov homeomorphisms of the surface which are not conjugate to powers of each other. Then there exists some integer $M$ such that the group $G = \langle f^m, g^n \rangle$ will be free for every $m,n> M$, and every element of this group will be a pseudo-Anosov. 

\end{theorem}

\begin{proof}
 Let $\psi, \phi \in \out$ be the fully irreducible, geometric outer automorphisms induced by $f,g$ respectively, where $\F = \pi_1(\mathcal{S})$. We will prove the result for $\psi, \phi\in \out$ which will imply the theorem.

 Let $[c]$ be the conjugacy class corresponding to $\partial\mathcal{S}$. Then $[c]$ fills $\F$ and $\psi([c]) = \phi([c]) = [c] $ and $ \mathcal{A}_{na}\Lambda^\pm_\psi = \mathcal{A}_{na}\Lambda^\pm_\phi = [\langle c \rangle]$. 
Proposition \ref{pingpong}  along with proposition \ref{duality} tells us that there exists an integer $M$ such that every cyclically reduced word $\xi$ in the group $G$ will be exponentially growing with a dual lamination pair $\Lambda^\pm_\xi$ such that any conjugacy class carried by $\mathcal{A}_{na} \Lambda^\pm_\xi$ will be carried by both $\mathcal{A}_{na}\Lambda^\pm_\psi$ and $\mathcal{A}_{na}\Lambda^\pm_\phi$. Hence we can conlude that $\mathcal{A}_{na}\Lambda^\pm_\xi = [\langle c \rangle]$.
We check that we satisfy all the hypothesis required for the main theorem except the last two bullets. Proposition \ref{duality} voids the need for bullet five. We modify the proof of Theorem \ref{main} by using Lemma \ref{mcg} so that we do not need the last bullet.
Using Proposition \ref{free} we can conclude that by increasing $M$ if necessary, we may assume that $G$ is free of rank two.

The proof of being being fully irreducible follows the exact same steps, but in this case we use Lemma \ref{mcg} to start the contradiction argument as in the proof of Theorem \ref{main}. Namely, assume that there does not exist any $M$ such that every element of $G$ is fully irreducible and  let $\xi_M\in G$ be a reducible element for each $M$. After passing to a subsequence if necessary assume that

$\xi_M=\psi^{m_1}\phi^{\epsilon'_1 n_1}........\psi^{\epsilon_k m_k}\phi^{n_k}$ \hspace{2pc}where $m_i=m_i(M)>0,n_j=n_j(M)>0$ and $k$ depend on $M$. From above discussion we have $\mathcal{A}_{na}\Lambda^\pm_{\xi_M} = [\langle c \rangle ]$, where $\mathcal{A}_{supp}[c]$ is not proper. Using Lemma \ref{mcg}, $\mathcal{A}_{supp}\Lambda^\pm_{\xi_M}$ must be proper and non trivial for all $M$. 
The rest of the argument follows through except that when we look at the part of the proof of Theorem \ref{main} separated under \textquotedblleft proof of claim \textquotedblright, the proof breaks down. We will just focus on this part and modify the proof to finish our theorem.

The goal of that part of the proof is to show that the edgelet length of the natural edges of $K_M$ is uniformly bounded. Suppose the claim is false. Then there exists a sequence of natural edges $E_M$ whose edgelet length goes to infinity. Let $l$ be some weak limit of this sequence. If we show $l$ is a periodic line $l = \rho_\infty$ then $\rho\subset E_M$ for all large  $M$, which implies that free factor support of $\rho$ is contained in the proper free factor $F_{\xi_M}$. The contradiction is achieved by showing that $[\rho]$ fills $\F$.

Let $\sigma$ be a weak limit of some subsequence of $\gamma^-_M$. We claim that $\sigma$ is not attracted to $\Lambda^+_\phi$. If not, then $\phi^p(\sigma)\in V^+_\phi$ for some $p\geq 0$. This means that for sufficiently large $M$, $\phi^p(\gamma^-_M)\in V^+_\phi$, implying that $\xi_M (\gamma^-_M)\in V^+_{\xi_M}$ for sufficiently large $M$, which is a contradiction to the fact that a generic leaf of $\Lambda^-_{\xi_M}$ is not attracted to $\Lambda^+_{\xi_M}$ under action of $\xi_M$. By similar arguments we can show that if $\sigma'$ is a weak limit of some sequence of $\gamma^+_M$, then $\sigma'$ is not attracted to $\Lambda^-_{\psi}$.

Let $l\in cl(\sigma)\cap cl(\sigma')$. Then $l\in \mathcal{B}_{na}(\Lambda^+_\phi)\cap \mathcal{B}_{na}(\Lambda^-_\psi)$ by above arguments and by Lemma \ref{NAL}. If $l$ is not carried by $\mathcal{A}_{na}(\Lambda^\pm_\psi)$, then by the weak attraction theorem (Lemma \ref{WAT}) $l$ is contained in every attracting neighborhood of the generic leaf $\lambda^+_\psi$.  This implies that $\lambda^+_\psi\subset cl(l) \subset cl(\sigma) \subset \mathcal{B}_{na}(\Lambda^+_\phi) $. But this contradicts our hypothesis that $\Lambda^+_\psi$ is attracted to $\Lambda^+_\phi$. Hence $l$ must be carried by 
$\mathcal{A}_{na}(\Lambda^\pm_\psi) = [\langle c \rangle]$. Hence $l = c_\infty$ is a periodic line and $[c]$ fills $\F$ and we get our contradiction by taking $c=\rho$. 

\end{proof}

\textbf{Example: }
Consider the following automorphisms $\phi$ and $\psi$.
$$\phi: a \mapsto a, b \mapsto Abaca, c\mapsto bacaB $$
$$\psi: a \mapsto ab, b \mapsto ac, c\mapsto a $$
Note that $\phi$ is a reducible, exponentially growing outer automorphism and $\psi$ is a hyperbolic, fully irreducible outer automorphism. A relative train track map for $\phi$
is given by

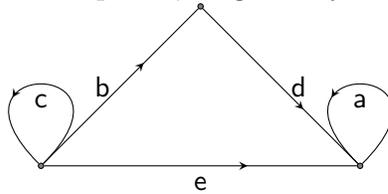
\begin{tikzpicture}[ node distance=3cm]

\usetikzlibrary{arrows}
\usetikzlibrary{decorations.markings}

\tikzstyle{every node}=[circle, draw, fill=black!50,
                        inner sep=0pt, minimum width=2pt]

\tikzstyle arrowstyle=[->, scale=1]

\tikzstyle directed=[postaction={decorate,decoration={markings,
    mark=at position .65 with {\arrow[arrowstyle]{stealth};}}}]

\node (1) {};
  \node (2) [below left of=1] {};
  \node (3) [below right of=1] {};

\path[every node/.style={font=\sffamily\small}]
    (1) edge [directed] node [right] {d} (3)  
    (2) edge [directed] node [below] {e} (3)
    (2) edge [directed]  node [left] {b} (1)
    (2) edge [directed, left, min distance =2cm] node [below] {c} (2)
    (3) edge [directed, min distance = 2cm, right,] node [below] {a} (3);
   
\end{tikzpicture}

 $$ a\mapsto a, b \mapsto bdaEc, c \mapsto bdaEceaDB, d\mapsto ea, e \mapsto ea $$

 The stratum are given by $H_1=\{a\}, H_2=\{e\}, H_3=\{d\}, H_4=\{c, b\}$
 
 We leave it to the reader to verify that the exponentially growing strata is  $H_4=\{c, b\}$ and the nonattracting subgroup system of the corresponding attracting lamination (in this example it is unique) is 
 $\mathcal{A}_{na}(\Lambda^\pm_\phi) = \{[\langle a \rangle]\}$. This implies that the corresponding attracting lamination is nongeometric. $\psi$ being a hyperbolic fully-irreducible outer automorphism, the unique 
 dual lamination pair $\Lambda^\pm_\psi$ is nongeometric and its nonattracting subgroup system is trivial.
 
 So, $\phi$ and $\psi$ satisfies all necessary conditions to apply Theorem \ref{main}, and so there exists some $M > 0$ such that for all $m,n \geq M$ the group $\langle \phi^m, \psi^n \rangle$ is a free group 
 of rank two and every element not conjugte to power of generators is hyperbolic and fully irreducible. It is to be noted that getting a precise value of $M$ is not easy since it depends on a lot of factors. However, 
 after extensively checking with the software developed by Thierry Coulbois for computing train track maps, it seems like $M=2$ might work in this case. It would be very intersting if we could understand if  $M$ is 
 dependent only on the rank of the $\F$.


\section{Relativized version of main theorem}

For the sake of keeping this work concise, we will only briefly go through the definitions and the reader is requested to refer to the work of Handel and 
Mosher titled ``Subgroup decomposition in $\out$'' (part IV in particular \citep{HM-13d}). 

Let $\mathcal{F}$ denote a free factor system that is left invariant by $\phi \in \out$. Then $\phi$ is said to be \textit{fully-irreducible relative to $\mathcal{F}$} if there does not exist any $\phi$-periodic free factor system $\mathcal{F}'$ such that $\mathcal{F} \sqsubset \mathcal{F}'$ and $\mathcal{F} \ne \mathcal{F}'$.

From Handel-Mosher's work on loxodromic elements of the free splitting complex \citep{HM-14b} one defines the concept of a \textit{co-edge} number for a free factor system $\mathcal{F}$ : it is an integer $\geq 1$ which is the minimum 
, over all subgraphs $H$ of a marked graph $G$ such that H realizes $\mathcal{F}$, of the number of edges in $G-H$ . Lemma 4.8 in \citep{HM-14b} gives an explicit formula for computing the co-edge 
number for a given free factor system.

Using relative train track theory one proves that if $\mathcal{F}$ has a co-edge number greater than or equal to 2, then the following are equivalent:

\begin{enumerate}
\label{relind}
 \item $\phi$ is fully irreducible relative to $\mathcal{F}$ (abbreviated as rel $\mathcal{F}$).
 \item There exists a dual lamination pair $\Lambda^\pm$ for $\phi$ such that $\Lambda^\pm$ fills relative to $\mathcal{F}$, and such that either $\Lambda^\pm$ is nongeometric and its 
 nonattracting subgroup system is simply $\mathcal{F}$ or $\Lambda^\pm$ is geometric and its nonattracting subgroup system is $\mathcal{F}$ plus another infinite cyclic component that together with $\mathcal{F}$ fills.
 Furthermore, the lamination pair $\Lambda^\pm$ is uniquely determined by this condition.
 
 \end{enumerate}
 
 \textbf{Remark:} $\Lambda^\pm$ \textit{fills relative to} $\mathcal{F}$ simply means that $\mathcal{F}_{supp}(\mathcal{F}, \Lambda^\pm) = [\mathbb{F}_N]$. Please note that the above equivalence is false if the 
 co-edge number is equal to 1.

 We now state the Relativized version of the pingpong lemma. This lemma is a modification of the pingpong proposition (proposition 1.3) in \citep{HM-13d} and the proof is exactly in the similar lines as the proof we have for 
 \ref{pingpong}, where the absolute version is just replaced by the relative versions and so will skip the details of the proof here.
 
 \textbf{Notation and setup:} Let $\mathcal{F}$ be a free factor system that is left invariant by $\phi, \psi \in \out$. Let $\Lambda^\pm_\phi, \Lambda^\pm_\psi$ be invariant dual lamination pairs for $\phi, \psi$ respectively.

 \begin{proposition}
  Suppose that the laminations $\Lambda^\pm_\phi, \Lambda^\pm_\psi$ each have a generic leaf $\lambda^\pm_\phi, \lambda^\pm_\psi$ which is fixed by 
  $\phi^\pm, \psi^\pm$ respectively, with fixed orientation. Also assume that the following conditions hold:
  
  \begin{itemize}
   \item $\Lambda^\pm_\phi$ is weakly attracted to $\Lambda^\epsilon_\psi$ under iterates of $\psi^\epsilon$ (where $\epsilon = +, -$).
   \item $\Lambda^\pm_\psi$ is weakly attracted to $\Lambda^\epsilon_\phi$ under iterates of $\phi^\epsilon$ (where $\epsilon = +, -$).
   \item $\mathcal{F} \sqsubset \mathcal{A}_{na}(\Lambda^\pm_\phi)$ and $\mathcal{F} \sqsubset \mathcal{A}_{na}(\Lambda^\pm_\psi)$
   \item Either both the lamination pairs $\Lambda^\pm_\psi, \Lambda^\pm_\psi$ are non-geometric or the subgroup $\langle \phi, \psi \rangle$ 
	    is geometric above $\mathcal{F}$
  \end{itemize}

 Then there exist attracting neighborhoods $V^\pm_\phi, V^\pm_\psi$ of $\Lambda^\pm_\phi, \Lambda^\pm_\psi$ respectively, 
 and there exists an integer $M$,  such that for every pair of finite sequences $n_i\geq M$ and $m_i\geq M$ if

 $$\xi= \psi^{\epsilon_1 m_1}\phi^{\epsilon_1 ' n_1}........\psi^{\epsilon_k m_k}\phi^{\epsilon_k ' n_k}$$ 
$(k\geq 1)$ is a cyclically reduced word then $\xi$ will be exponentially-growing and have a lamination pair $\Lambda^\pm_\xi$ satisfying the following properties:
  
  \begin{enumerate}
   \item $\Lambda^\pm_\xi$ is non-geometric if $\Lambda^\pm_\phi$ and $\Lambda^\pm_\psi$ are both non-geometric.
   \item $\mathcal{F}$ is carried by $\mathcal{A}_{na}(\Lambda^\pm_\xi)$ and  $\mathcal{A}_{na}(\Lambda^\pm_\xi)$ is caried by  $\mathcal{A}_{na}(\Lambda^\pm_\phi)$ and  
   $\mathcal{A}_{na}(\Lambda^\pm_\psi)$.
    \item $\psi^{m_i}(V^\pm_\phi) \subset V^+_\psi$ and $\psi^{-m_i}(V^\pm_\phi) \subset V^-_\psi$ .
  \item $\phi^{n_j}(V^\pm_\psi) \subset V^-_\phi$ and $\phi^{-n_j}(V^\pm_\psi) \subset V^-_\phi$ .
  \item $V^+_\xi \colon= V^{\epsilon_1}_\psi$ is an attracting neighborhood of $\Lambda^+_\xi$ 
  \item $V^-_\xi \colon= V^{-\epsilon_k '}_\phi$ is an attracting neighborhood of $\Lambda^-_\xi$ 
  \item (uniformity) Suppose $U^{\epsilon_1}_\psi$ is an attracting neighborhood of $\Lambda^{\epsilon_1}_\psi$ then some generic leaf of $\Lambda^+_\xi$ belongs to $U^{\epsilon_1}_\psi$ for 
  sufficiently large $M$.
  \item (uniformity) Suppose $U^{\epsilon_k'}_\phi$ is an attracting neighborhood of $\Lambda^{\epsilon_k'}_\phi$ then some generic leaf of $\Lambda^+_\xi$ belongs to $U^{\epsilon_k'}_\phi$ for 
  sufficiently large $M$.
  \end{enumerate}

 \end{proposition}

\begin{definition}
 Let $\phi, \psi \in \out$ be exponentially growing outer automorphisms with invariant lamination pairs $\Lambda^\pm_\phi, \Lambda^\pm_\psi$ and let $\mathcal{F}$ be a free factor system 
 which is left invariant by both $\phi, \psi$. Suppose the following conditions hold:
 \begin{enumerate}
  \item None of the lamination pairs $\Lambda^\pm_\phi, \Lambda^\pm_\psi$ are carried by $\mathcal{F}$.
  \item $\{\Lambda^\pm_\phi\} \cup \{\Lambda^\pm_\psi\}$ fill relative to $\mathcal{F}$.
  \item $\Lambda^\pm_\phi$ is weakly attracted to $\Lambda^\epsilon_\psi$ under iterates of $\psi^\epsilon$ (where $\epsilon = +, -$).
  \item $\Lambda^\pm_\psi$ is weakly attracted to $\Lambda^\epsilon_\phi$ under iterates of $\phi^\epsilon$ (where $\epsilon = +, -$).
  \item $\mathcal{A}_{na}(\Lambda^\pm_\phi)$ and $\mathcal{A}_{na}(\Lambda^\pm_\psi)$ are mutually malnormal relative to $\mathcal{F}$.
  \item Either both the lamination pairs $\Lambda^\pm_\psi, \Lambda^\pm_\psi$ are non-geometric or the subgroup $\langle \phi, \psi \rangle$ 
	    is geometric above $\mathcal{F}$
 \end{enumerate}

 In this case we define the pair $(\phi, \Lambda^\pm_\phi) , (\psi, \Lambda^\pm_\psi)$ to be independent \textit{relative to $\mathcal{F}$}.
\end{definition}

\textbf{Remark:} Here the term \textit{mutually malnormal relative to $\mathcal{F}$} in condition 5 above, means that a line or a conjugacy class is carried by both 
$\mathcal{A}_{na}(\Lambda^\pm_\phi)$ and $\mathcal{A}_{na}(\Lambda^\pm_\psi)$ if and only if it is carried by $\mathcal{F}$.

 \begin{theorem}
  Given a free factor system $\mathcal{F}$ with co-edge number $\geq 2$, given $\phi, \psi \in \out$ each preserving $\mathcal{F}$, and given invariant lamination pairs 
  $\Lambda^\pm_\phi, \Lambda^\pm_\psi$, so that the pair $(\phi, \Lambda^\pm_\phi), (\psi, \Lambda^\pm_\psi)$ is independent reltive to $\mathcal{F}$, then there $\exists$  $M\geq 1$, 
  such that for any integer $m,n \geq M$, the group $\langle \phi^m, \psi^n \rangle$ is a free group of rank 2, all of whose non-trivial elements except perhaps the powers of $\phi, \psi$ 
  and their conjugates, are fully irreducible relative to $\mathcal{F}$.
 \end{theorem}

We will only give a brief sketch of the proof, going over the key points. The full details of the proof are very similar to the proof of our main theorem \ref{main} 
  and most of it is available in \citep{HM-13d} in section 2.
 
 \begin{proof}
  The conclusion about the rank 2 free group follows easily from the Tit's Alternative work of Bestvina-Feighn-handel \citep{BFH-00} Lemma 3.4.2, which gives us some integer $M_0$ such that 
  for every $m,n \geq M_0$ the group $\langle \phi^m, \psi^n \rangle$ is a free group of rank 2.
  
  For the conclusion about being fully irreducible relative to $\mathcal{F}$, suppose that the conclusion is false.

	$\Longrightarrow$ For every $M \geq M_0$, there exists some $m(M), n(M)$ such that the group $\langle \phi^m, \psi^n \rangle$ contains at least one 
	non trivial reduced word $\xi_M$ which is not the powers of generators themselves or their conjugates, and $\xi_M$ is not fully irreducible relative to $\mathcal{F}$.

	Next, using the conclusions from our relativized pingpong lemma earlier in this section, $\mathcal{F}$ is carried by $\mathcal{A}_{na}(\Lambda^\pm_\xi)$ and  $\mathcal{A}_{na}(\Lambda^\pm_\xi)$ is caried by  $\mathcal{A}_{na}(\Lambda^\pm_\phi)$ and  
   $\mathcal{A}_{na}(\Lambda^\pm_\psi)$.
   
	$\Longrightarrow$ $\mathcal{A}_{na}(\Lambda^\pm_{\xi_M}) = \mathcal{F}$.
	
	$\Longrightarrow$ $\Lambda^\pm_{\xi_M}$ is non-geometric.
	
	Also the additional co-edge $\geq 2$ condition tells us that $\Lambda^\pm_{\xi_M}$ cannot fill relative to $\mathcal{F}$. This means that the 
	free factor system $\mathcal{F}_M  = \mathcal{F}_{supp}(\mathcal{F}, \Lambda^\pm_{\xi_M})$ is a proper free factor system.
	
	Now we proceed exactly in the same fashion as the proof of theorem I in the non-geometric case goes in section 2.4 of \citep{HM-13d}. The key idea is to use 
	stallings graph to drive up $\mathcal{F}_M $ and arrive at a contradiction exactly in a similar fashion to the proof of our main theorem.
	
	This is acheived in the proof by showing that if $M$ is sufficiently large then, we have $\mathcal{F}_\phi , \mathcal{F}_\psi \sqsubset \mathcal{F}_M$. 
	This implies that $\mathcal{F}_{supp}(\mathcal{F}_\phi, \mathcal{F}_\psi)$ is proper. But this contradicts the condition that $\{\Lambda^\pm_\phi\} \cup \{\Lambda^\pm_\psi\}$ fill.

 \end{proof}

 We end this section with a corollary that is a direct application of the above theorem.

 \begin{corollary}
 Given a free factor system $\mathcal{F}$ with coedge number $\geq 2$, and given $\phi, \psi \in \out$ , if $\phi, \psi$ are fully irreducible relative 
 to $\mathcal{F}$, with corresponding invariant lamination pairs $\Lambda^\pm_\phi, \Lambda^\pm_\psi$ (as in the equivalence condition \ref{relind})such that the pair $\{\Lambda^+_\phi , \Lambda^-_\phi\}$ 
 is disjoint from the pair 
 $\{\Lambda^+_\psi, \Lambda^-_\psi\}$, then there exists an integer $M \geq 1$ such that for any $m,n \geq M$ the group $\langle \phi^m, \psi^n\rangle$ is a free group 
 of rank 2 and every element of this group is fully irreducible relative to $\mathcal{F}$.
 \end{corollary}
 
 \begin{proof}
  The only thing to observe here is the disjoint property of the lamination sets makes it satisfy all the conditions of the definition of relative independence, except the technical condition that gives us duality of 
  the lamination pairs. However, as we have noted, the lamination for an element which is fully irreducible relative to $\mathcal{F}$, we can directly prove the 
  duality of laminations arising from pingpong in similar lines as in \ref{duality}.
 \end{proof}

%
\bibliographystyle{plainnat}
\bibliography{biblo_R1}

\begin{thebibliography}{18}
\providecommand{\natexlab}[1]{#1}
\providecommand{\url}[1]{\texttt{#1}}
\expandafter\ifx\csname urlstyle\endcsname\relax
  \providecommand{\doi}[1]{doi: #1}\else
  \providecommand{\doi}{doi: \begingroup \urlstyle{rm}\Url}\fi

\bibitem[Bestvina and Feighn(1992)]{BF-92}
Mladen Bestvina and Mark Feighn.
\newblock {A combination theorem for negetively curved groups}.
\newblock \emph{J. Differential Geom.}, 43\penalty0 (4):\penalty0 85--101,
  1992.

\bibitem[Bestvina and Feighn(1997)]{BFH-97}
Mladen Bestvina and Mark Feighn.
\newblock {Laminations, {T}rees, and irreducible automorphisms of free groups}.
\newblock \emph{Geom. Func. Anal.}, 7:\penalty0 215--244, 1997.

\bibitem[Bestvina and Feighn(2014)]{BF-12}
Mladen Bestvina and Mark Feighn.
\newblock {Hyperbolicity of the complex of free factors}.
\newblock \emph{Advances in Mathematics}, 256:\penalty0 104--155, 2014.
\newblock URL \url{http://arxiv.org/abs/1107.3308}.

\bibitem[Bestvina and Handel(1992)]{BH-92}
Mladen Bestvina and Michael Handel.
\newblock {Train tracks and {A}utomorphisms of {F}ree groups}.
\newblock \emph{Ann. of Math.}, 135:\penalty0 1--51, 1992.

\bibitem[Bestvina et~al.(2000)Bestvina, Feighn, and Handel]{BFH-00}
Mladen Bestvina, Mark Feighn, and Michael Handel.
\newblock {Tits alternative in {O}ut({F}n) -{I}: {D}ynamics of
  {E}xponentially-growing {A}utomorphisms}.
\newblock \emph{Ann. of Math.}, 151:\penalty0 517--623, 2000.

\bibitem[Brinkmann(2000)]{Br-00}
Peter Brinkmann.
\newblock {Hyperbolic Automorphisms of Free Groups}.
\newblock \emph{Geom. Funct. Anal.}, 10\penalty0 (5):\penalty0 1071--1089,
  2000.

\bibitem[Clay and Pettet(2010)]{CP-10}
Matt Clay and A~Pettet.
\newblock {Twisting out fully irreducible automorphisms}.
\newblock \emph{Geom. Funct. Anal.}, 20\penalty0 (3):\penalty0 657--689, 2010.

\bibitem[Farb and Mosher(1992)]{FaM-02}
Benson Farb and Lee Mosher.
\newblock {Convex {C}ocompact subgroups of {Ma}pping class groups}.
\newblock \emph{Geom. and Top.}, 6:\penalty0 91--152, 1992.

\bibitem[Feighn and Handel(2011)]{FH-11}
Mark Feighn and Michael Handel.
\newblock {The recognition theorem for {O}ut({F}n)}.
\newblock \emph{Groups Geom. Dyn}, 5:\penalty0 39--106, 2011.

\bibitem[Gaboriau et~al.(1998)Gaboriau, Levitt, Jaeger, and Lustig]{GJLL-98}
D~Gaboriau, G~Levitt, A~Jaeger, and M~Lustig.
\newblock {An index for counting fixed points of automorphisms of free groups}.
\newblock \emph{Duke Math. J}, 93\penalty0 (3):\penalty0 425--452, 1998.

\bibitem[Gultepe(2015)]{Gul}
Funda Gultepe.
\newblock { Fully irreducible Automorphisms of the Free Group via Dehn twisting
  in $\sharp_k(S_2\times S_1)$}.
\newblock \emph{math arxiv}, arXiv:1411.7668, 2015.
\newblock URL \url{http://arxiv.org/abs/1411.7668}.

\bibitem[Handel and Mosher(2009)]{HM-09}
Michael Handel and Lee Mosher.
\newblock {Subgroup classification in Out(F\_n)}.
\newblock \emph{arXiv}, arXiv:0908.1255, 2009.
\newblock URL \url{http://arxiv.org/abs/0908.1255}.

\bibitem[Handel and Mosher(2013{\natexlab{a}})]{HM-13c}
Michael Handel and Lee Mosher.
\newblock {Subgroup decomposition in Out(F\_n), Part III: Weak Attraction
  theory}.
\newblock \emph{arXiv}, arXiv:0908.1255, 2013{\natexlab{a}}.
\newblock URL \url{http://arxiv.org/find/math/1/au:+Mosher_L/0/1/0/all/0/1}.

\bibitem[Handel and Mosher(2013{\natexlab{b}})]{HM-13d}
Michael Handel and Lee Mosher.
\newblock {Subgroup decomposition in Out(F\_n) IV: Relatively irreducible
  subgroups}.
\newblock \emph{arXiv}, arXiv:1306.4712, 2013{\natexlab{b}}.

\bibitem[Handel and Mosher(2014)]{HM-14b}
Michael Handel and Lee Mosher.
\newblock {The free splitting complex of a free group II: Loxodromic outer
  automorphisms}.
\newblock \emph{arXiv}, arXiv:1402.1886, 2014.

\bibitem[Kapovich and Lustig(2010)]{KL-10}
Ilya Kapovich and Martin Lustig.
\newblock {Ping-pong and Outer space}.
\newblock \emph{J. Topl. Anal}, 2\penalty0 (2):\penalty0 173--201, 2010.

\bibitem[Mann(2014)]{Mann}
Brian Mann.
\newblock {Hyperbolicity of the Cyclic Splitting Graph}.
\newblock \emph{Geometriae Dedicata}, 173\penalty0 (1):\penalty0 271--280,
  2014.

\bibitem[Stalings(1983)]{St-83}
John Stalings.
\newblock {Topology of finite graphs}.
\newblock \emph{Invent. Math.}, 71\penalty0 (3):\penalty0 551--565, 1983.

\end{thebibliography}
\end{document}